\documentclass{amsart}
\usepackage{amssymb,amsmath,amsthm,amsfonts,amsopn,url,bbm, hyperref}
\usepackage{enumitem}
\theoremstyle{plain}
\newtheorem{thm}{Theorem}[subsection]

\newtheorem{prop}[thm]{Proposition}
\newtheorem{lemma}[thm]{Lemma}

\newtheorem{conj}[thm]{Conjecture}
\theoremstyle{definition}
\newtheorem{defn}[thm]{Definition}
\newtheorem*{defn*}{Definition}
\newtheorem{example}[thm]{Example}
\newtheorem*{example*}{Example}
\newtheorem{construction}[thm]{Construction}
\theoremstyle{remark}
\newtheorem{rmk}[thm]{Remark}
\newtheorem*{rmk*}{Remark}
\newtheorem*{sociocomment}{Sociological comment}
\newcommand{\field}[1]{\mathbbm{#1}}
\newcommand{\C}{\field{C}}
\newcommand{\F}{\field{F}}
\newcommand{\N}{\field{N}}
\newcommand{\Q}{\field{Q}}

\newcommand{\Z}{\field{Z}}
\newcommand{\ideal}[1]{\mathfrak{#1}}
\newcommand{\m}{\ideal{m}}
\newcommand{\n}{\ideal{n}}
\newcommand{\p}{\ideal{p}}

\newcommand{\ia}{\ideal{a}}

\newcommand{\cal}{\mathcal}
\newcommand{\cA}{\cal{A}}
\newcommand{\cC}{\cal{C}}
\newcommand{\cF}{\cal{F}}
\newcommand{\cG}{\cal{G}}
\newcommand{\cI}{\cal{I}}
\newcommand{\cM}{\cal{M}}
\newcommand{\cP}{\cal{P}}
\newcommand{\cR}{\cal{R}}
\newcommand{\cS}{\cal{S}}
\newcommand{\cT}{\cal{T}}
\DeclareMathOperator{\eq}{\!eq}
\DeclareMathOperator{\h}{\!h}
\DeclareMathOperator{\fg}{\!fg}

\DeclareMathOperator{\Max}{Max}
\DeclareMathOperator{\Spec}{Spec}

\DeclareMathOperator{\depth}{depth}

\DeclareMathOperator{\Hom}{Hom}

\DeclareMathOperator{\Jac}{Jac}
\DeclareMathOperator{\chr}{char}
\DeclareMathOperator{\core}{core}
\newcommand{\bfc}{{\rm{bf}}}
\newcommand{\spc}{\rm{sp}}

\newcommand{\rr}[1]{\widetilde{{#1}}}

\newcommand{\ra}{\rightarrow}

\newcommand{\rc}{{\rm c}}
\newcommand{\rcl}{{\rm cl}}
\newcommand{\rd}{{\rm d}}
\newcommand{\repf}{{\rm epf}}
\newcommand{\rt}{{\rm t}}
\newcommand{\rv}{{\rm v}}
\newcommand{\rw}{{\rm w}}
\newcommand{\bd}{{\mathbf d}}
\newcommand{\br}{{\mathbf r}}

\author{Neil Epstein}
\address{Universit\"at Osnabr\"uck \\ 
Institut f\"ur Mathematik \\ 
49069 Osnabr\"uck \\ Germany}
 \email{nepstein@uni-osnabrueck.de}
 
\title{A guide to closure operations in commutative algebra}

\date{August 30, 2011}
\begin{document}
\begin{abstract}
This article is a survey of closure operations on ideals in commutative rings, with an emphasis on structural properties and on using tools from one part of the field to analyze structures in another part.  The survey is broad enough to encompass the radical, tight closure, integral closure, basically full closure, saturation with respect to a fixed ideal, and the v-operation, among others.
\end{abstract}

 \thanks{The author was partially supported by a grant from the DFG (German Research Foundation)}

\maketitle
\tableofcontents

\section{Introduction}

There have been quite a few books and survey articles on tight closure (e.g. \cite{HuTC, Hu-tcparam, Smintro, Sm-tcvan, Ho-tcharp}), on integral closure (e.g. \cite{Vasc-arith, Vasc-icbook, HuSw-book}), and on star-operations on integral domains (e.g. \cite{FoLo-Kronecker}, \cite[chapters 32 and 34]{Gil-MIT}), but as far as this author knows, no such article on closure operations in general.  However, several authors (e.g. \cite{nme-sp,  Br-Groth, Va-cl}) have recently found it useful to consider closure operations as a subject in itself, so I write this article as an attempt to provide an overall framework.  This article is intended both for the expert in one closure operation or another who wants to see how it relates to the rest, and for the lay commutative algebraist who wants a first look at what closure operations are.  For the most part, this article will not go into the reasons why any given closure operation is important.  Instead, I will concentrate on the structural aspects of closure operations, how closure operations arise, how to think about them, and how to analyze them.

The reader may ask: ``If the only closure operation I am interested in is $\rc$, why should I care about other closure operations?''  Among other reasons: the power of analogical thinking is central to what mathematicians do.  If the $\rd$-theorists have discovered or used a property of their closure operation $\rd$, the $\rc$-theorist may use this to investigate the analogous property for $\rc$, \emph{and may not have thought to do so without knowledge of $\rd$-closure}.  Morover, what holds trivially for one closure operation can be a deep theorem (or only hold in special cases) for another -- and vice versa.  A good example is persistence (see \ref{sub:persist}). 

In most survey articles, one finds a relatively well-defined subject and a more-or-less linear progression of ideas.  The subject exists as such in the minds of those who practice it before the article is written, and the function of the article is to introduce new people to the already extant system of ideas.  The current article serves a somewhat different function, as the ideas in this paper are not linked sociologically so much as axiomatically.  Indeed, there are at least three socially distinct groups of people studying these things, some of whom seem barely to be aware of each other's existence.  In this article, one of my goals is to bridge that gap.

The structure of the article is as follows: In \S \ref{sec:whatis}, I introduce the notion of closures, eleven typical examples, and some non-examples.  In \S\ref{sec:cons}, I exhibit six simple constructions and show how all the closure operations from \S\ref{sec:whatis} arise from these.  The next section, \S\ref{sec:props}, concerns properties that closures may have; it comprises more than 1/4 of the paper!  In it, we spend a good bit of time on star- and semi-prime operations, after which we devote a subsection each to forcing algebras, persistence, homological conjectures, tight closure-like properties, and (homogeneous) equational closures.  The short \S\ref{sec:red}  explores a tightly related set of notions involving what happens when one looks at the collection of \emph{sub}ideals that have the same closure as a given ideal.  In \S\ref{sec:rings}, we explore ring properties that arise from certain ideals being closed.  Finally in \S\ref{sec:modules}, we talk about various ways to extend to closures on (sub)\emph{modules}.  Beyond the material in \S\S \ref{sec:whatis} and \ref{sec:cons}, the reader may read the remaining sections in almost any order.

Throughout this paper, $R$ is a commutative ring with unity.  At this point, one would normally either say that $R$ will be assumed Noetherian, or that $R$ will not necessarily be assumed to be Noetherian.  However, one of the reasons for the gap mentioned above is that people are often scared off by such statements.  It is true that many of the examples I present here seem to work best (and are most studied) in the Noetherian context.   On the other hand, I have also included some of the main examples and constructions that are most interesting in the non-Noetherian case.  As my own training is among those who work mainly with Noetherian rings, it probably is inevitable that I will sometimes unknowingly assume a ring is Noetherian.  In any case, the article should remain accessible and interesting to all readers.

\section{What is a closure operation?}\label{sec:whatis}

\subsection{The basics}

\begin{defn}\label{def:closure}
Let $R$ be a ring.  A \emph{closure operation} $\rcl$ on a set of ideals $\cI$ of $R$ is a set map $\rm{cl}: \cI \ra \cI$ ($I \mapsto I^\rcl)$ satisfying the following conditions: \begin{enumerate}
\item\label{it:extend} (Extension) $I \subseteq I^\rcl$ for all $I \in \cI$.
\item\label{it:idem} (Idempotence) $I^\rcl= (I^\rcl)^\rcl$ for all $I \in \cI$.
\item\label{it:preserve} (Order-preservation) If $J \subseteq I$ are ideals of $\cI$, then $J^\rcl \subseteq I^\rcl$.
\end{enumerate}

If $\cI$ is the set of \emph{all} ideals of $R$, then we say that $\rcl$ is a closure operation \emph{on $R$}.

We say that an ideal $I \in \cI$ is \emph{$\rcl$-closed} if $I = I^\rcl$.
\end{defn}

As far as I know, this concept is due to E. H. Moore \cite{Moore-analysis}, who defined it (over a century ago!) more generally for subsets of a set, rather than ideals of a ring.  Moore's context was mathematical analysis.  His has been the accepted definition of ``closure operation'' in lattice theory and universal algebra ever since (e.g. \cite[V.1]{Bir-latbook} or \cite[7.1]{DaPr-latbook}).  Oddly, this general definition of closure operation does not seem to have gained currency in commutative algebra until the late 1980s \cite{Rat-Delta, OkRat-filtclosure}, although more special structures already had standard terminologies associated to them (see \ref{sub:starsemi}).

\begin{example}[Examples of closures]\label{ex:closures}  The reader is invited to find his/her favorite closure(s) on the following list.  Alternately, the list may be skipped and referred back to when an unfamiliar closure is encountered in the text.

\begin{enumerate}
\item The \emph{identity closure}, sending each ideal to itself, is a closure operation on $R$.  (In multiplicative ideal theory, this is usually called the \emph{$\rd$-operation}.)
\item The \emph{indiscrete closure}, sending each ideal to the unit ideal $R$, is also a closure operation on $R$.
\item The \emph{radical} is the first nontrivial example of a closure operation on an arbitrary ring $R$.  It may be defined in one of two equivalent ways.  Either \[
\sqrt{I} := \{f \in R \mid \exists \text{a positive integer } n \text{ such that } f^n \in I\}
\] or \[
\sqrt{I} := \bigcap \{\p \in \Spec R \mid I \subseteq \p\}.
\]
The importance of the radical is basic in the field of algebraic geometry, due to Hilbert's Nullstellensatz (cf. any introductory textbook on algebraic geometry).
\item Let $\ia$ be a fixed ideal of $R$.  Then \emph{$\ia$-saturation} is a closure operation on $R$.  Using the usual notation of $(- : \ia^\infty)$, we may define it as follows: \[
(I : \ia^\infty) := \bigcup_{n\in \N} (I : \ia^n) = \{r \in R \mid \exists n \in \N \text{ such that } \ia^n r \subseteq I \}
\]
This operation is important in the study of \emph{local cohomology}.  Indeed, $H^0_\ia(R/I) = \frac{(I : \ia^\infty)}{I}$.
\item The \emph{integral closure} is a closure operation as well.  One of the many equivalent definitions is as follows:  For an element $r\in R$ and an ideal $I$ of $R$,\footnote{Some may find my choice of notation surprising.  Popular notations for integral closure include $I_a$ and $\overline{I}$.  I avoid the first of these because it looks like a variable subscript, as the letter $a$ does not seem to stand for anything.  The problem with the second notation is that it is overly ambiguous.  Such notation can mean integral closure of rings, integral closure of ideals, a quotient module, and so forth.  So in my articles, I choose to use the $I^-$ notation to make it more consistent with the notation of other closure operations (such as tight, Frobenius, and plus closures).} $r \in I^-$ if there exist $n\in \N$ and $a_i \in I^i$ for $1\leq i \leq n$ such that \[
r^n + \sum_{i=1}^n a_i r^{n-i} = 0.
\]
Integral closure is a big topic.  See for instance the books \cite{HuSw-book, Vasc-icbook}.
\item Let $R$ be an integral domain.  Then \emph{plus closure} is a closure operation.  It is traditionally linked with tight closure (see below), and defined as follows:  For an ideal $I$ and an element $x\in R$, we say that $x \in I^+$ if there is some injective map $R \ra S$ of integral domains, which makes $S$ a finite $R$-module, such that $x \in IS$.  (See \ref{sub:fromcon}(6) for general $R$.)
\item Let $R$ be a ring of prime characteristic $p>0$.  Then \emph{Frobenius closure} is a closure operation on $R$.  To define this, we need the concept of bracket powers.  For an ideal $I$, $I^{[p^n]}$ is defined to be the ideal generated by all the $p^n$th powers of elements of $I$.  For an ideal $I$ and an element $x\in R$, we say that $x\in I^F$ if there is some $n\in \N$ such that $x^{p^n} \in I^{[p^n]}$.
\item Let $R$ be a ring of prime characteristic $p>0$.  Then \emph{tight closure} is a closure operation on $R$.  For an ideal $I$ and an element $x\in R$, we say that $x\in I^*$ if there is some power $e_0 \in \N$ such that the ideal $\bigcap_{e\geq e_0} (I^{[p^e]} : x^{p^e})$ is not contained in any minimal prime of $R$.
\item Let $R$ be a complete local domain.  For an $R$-algebra $S$, we say that $S$ is \emph{solid} if $\Hom_R(S,R) \neq 0$.  We define \emph{solid closure} on $R$ by saying that $x\in I^\star$ if $x\in IS$ for some solid $R$-algebra $S$.  (See \ref{sub:fromcon}(9) for general $R$.)
\item Let $\Delta$ be a multiplicatively closed set of ideals.  The \emph{$\Delta$-closure} \cite{Rat-Delta} of an ideal $I$ is $I^\Delta := \bigcup_{K \in \Delta} (IK :K)$.  Ratliff \cite{Rat-Delta} shows close connections between $\Delta$-closure and integral closure for appropriate choices of $\Delta$.
\item If $(R,\m)$ is local, and $I$ is $\m$-primary, then the \emph{basically full closure} \cite{HRR-bf} of $I$ is $I^\bfc := (I \m : \m)$.  (Note: This is a closure operation even for non-$\m$-primary ideals $I$.  However, only for $\m$-primary $I$ does it produce the smallest so-called ``basically full'' ideal that contains $I$.)
\end{enumerate}

Additional examples of closures include the $\rv$-, $\rt$-, and $\rw$-operations (\ref{sub:starsemi}), various tight closure imitators (\ref{sub:tcim}),  continuous and axes closures \cite{Br-cc}, natural closure \cite{nmeHo-cc}, and weak subintegral closure \cite{VitLe-wsi}. (See the references for more details on these last four.)
\end{example}

Some properties follow from the axiomatic definition of a closure operation:

\begin{prop}\label{pr:basics}
Let $R$ be a ring and $\rcl$ a closure operation. Let $I$ be an ideal and $\{I_\alpha\}_{\alpha \in \cA}$ a set of ideals. \begin{enumerate}
\item If every $I_\alpha$ is $\rcl$-closed, so is $\bigcap_\alpha I_\alpha$.
\item $\bigcap_\alpha {I_\alpha}^\rcl$ is $\rcl$-closed.
\item $I^\rcl$ is the intersection of all $\rcl$-closed ideals that contain $I$.
\item $\left(\sum_\alpha {I_\alpha}^\rcl\right)^\rcl = \left(\sum_\alpha I_\alpha\right)^\rcl$.
\end{enumerate}
\end{prop}

\begin{proof}
Let $I$ and $\{I_\alpha\}$ be as above.
\begin{enumerate}
\item For any $\beta \in \cA$, we have $\bigcap_\alpha I_\alpha \subseteq I_\beta$, so since $\rcl$ is order-preserving, we have $\left(\bigcap_\alpha I_\alpha\right)^\rcl \subseteq {I_\beta}^\rcl = I_\beta$.  Since this holds for any $\beta$, we have $\left(\bigcap_\alpha I_\alpha\right)^\rcl \subseteq \bigcap_\beta I_\beta = \bigcap_\alpha I_\alpha$.

\item This follows directly from part (1).

\item Let $J$ be an ideal such that $I \subseteq J = J^\rcl$.  Then by order-preservation, $I^\rcl \subseteq J^\rcl = J$, so $I^\rcl$ is contained in the given intersection.  But since $I^\rcl = (I^\rcl)^\rcl$ is one of the ideals being intersected, the conclusion follows. 

\item `$\supseteq$': By the Extension property,  ${I_\alpha}^\rcl \supseteq I_\alpha$, so $\sum_\alpha {I_\alpha}^\rcl \supseteq \sum_\alpha I_\alpha$.  Then the conclusion follows from order-preservation.

`$\subseteq$': For any $\beta \in \cA$, $I_\beta \subseteq \sum_\alpha I_\alpha \subseteq \left(\sum_\alpha I_\alpha\right)^\rcl$, so by order-preservation and idempotence, ${I_\beta}^\rcl \subseteq  \left(\left(\sum_\alpha I_\alpha\right)^\rcl\right)^\rcl =  \left(\sum_\alpha I_\alpha\right)^\rcl$.  Since this holds for all $\beta \in \cA$, we have $\sum_\alpha {I_\alpha}^\rcl \subseteq \left(\sum_\alpha I_\alpha\right)^\rcl$.

\end{enumerate}
\end{proof}

We finish the subsection on ``basics'' by giving two alternate characterizations of closure operations on $R$:

\begin{rmk}
Here is a ``low-tech'' way of looking at closure operations, due essentially to Moore \cite{Moore-analysis}.  Namely, giving a closure operation is \emph{equivalent} to giving a collection $\cC$ of ideals such that the intersection of any subcollection is also in $\cC$.

For suppose $\rcl$ is a closure operation on $R$.  Let $\cC$ be the class of $\rcl$-closed ideals.  That is, $I \in \cC$ iff $I = I^\rcl$.  By Proposition~\ref{pr:basics}(1), the intersection of any subcollection of ideals in $\cC$ is also in $\cC$.

Conversely, suppose $\cC$ is a collection of ideals for which the intersection of any subcollection is in $\cC$.  For an ideal $I$, let $I^\rcl := \bigcap \{J \mid I \subseteq J \in \cC \}$.  All three of the defining properties of closure operations follow easily.  Hence, $\rcl$ is a closure operation.

The applicability of this observation is obvious: Given any collection of ideals in a ring, one may obtain a closure operation from it by extending it to contain all intersections of the ideals in the collection, and letting these be the closed ideals.  The resulting closure operation may then be used to analyze the property that defined the original class of ideals.
\end{rmk}

\begin{rmk}
On the other hand, here is a ``high-tech'' way of looking at closure operations.  Let $R$ be a ring, and $\cC$ be the \emph{category} associated to the partially ordered set of ideals of $R$.  Then a closure operation on $R$ is the same thing as a \emph{monad in the category $\cC$}. (see \cite[VI.1]{Mac-CWM} for the definition of monad in a category)  It is easy to see that any monad in a poset is idempotent, and the theory of idempotent monads is central in the study of so-called ``localization functors'' in algebraic topology (thanks to G. Biedermann and G. Raptis, who mention \cite[Chapter 2]{Ad-loc} as a good source, for pointing me in this direction).
\end{rmk}

\subsection{Not-quite-closure operations}
It should be noted that the three given axioms of closure operations are independent of each other; many operations on ideals satisfy two of the axioms without satisfying the third.  For example, the operation on ideals that sends every ideal to the $0$ ideal is idempotent (condition (2)) and is order-preserving (condition (3)), but of course is not extensive unless $R$ is the zero ring.

For an operation that is extensive (1) and order-preserving (3), but is not idempotent, let $f$ be a fixed element (or ideal) of $R$, and consider the operation $I \mapsto (I :f)$.  This is almost never idempotent.  For example one always has $((f^2 :f):f) = (f^2 : f^2)=R$ of course, but if $f$ is any nonzero element of the Jacobson radical of $R$, then $(f^2 :f) \neq R$.

Another non-idempotent operation that is extensive and order-preserving is the so-called ``$\ia$-tight closure", for a fixed ideal $\ia$ \cite{HaYo-atc}, denoted $(\ )^{*\ia}$.   By definition, $x \in I^{*\ia}$ if there is some $e_0 \in \N$ such that the ideal $\bigcap_{e \geq e_0} (I^{[p^e]} : \ia^{p^e} x^{p^e})$ is not contained in any minimal prime.  In their Remark 1.4, Hara and Yoshida note that if $\ia=(f)$ is a principal ideal, then $I^{*\ia} = (I^* : \ia)$, an operation which we have already noted fails to be idempotent.  (For a similarly-defined operation which actually \emph{is} a closure operation, see \cite{Vr-atc}.)

Consider the operation which sends each ideal $I$ to its \emph{unmixed part} $I^{\rm unm}$ \cite{Hu-tcparam}.  This is defined by looking at the primary ideals (commonly called \emph{components}) in an irredundant minimal primary decomposition of $I$, and then intersecting those components that have \emph{maximum dimension}.  Although the decomposition is not uniquely determined, the components of maximum dimension are, so this is a well-defined operation.  Moreover,  this operation is extensive (1) and idempotent (2) (since all the components of $I^{\rm unm}$ already have the same dimension), but is \emph{not order-preserving} in general.  For an example, let $R=k[x,y]$ be a polynomial ring in two variables over a field $k$, and let $J := (x^2, xy)$ and $I := (x^2, xy, y^2)$.  Then $J \subseteq I$, but $J^{\rm unm} = (x) \nsubseteq I^{\rm unm}=I$.  Similar comments apply in the 3-variable case when $J = (xy, xz)$ and $I=(y,z)$.

For another extensive, idempotent operation which is not order-preserving, consider the ``Ratliff-Rush closure'' (or ``Ratliff-Rush operation''), given in \cite{RatRu-rr}, defined on so-called \emph{regular ideals} (where an ideal $I$ of $R$ is \emph{regular} if it contains an $R$-regular element), defined by $\tilde{I} := \bigcup_{n=1}^\infty (I^{n+1} : I^{n})$.  In  \cite[1.11]{HLS-rr} (resp. in \cite[1.1]{HJLS-coeff}), the domain $R := k[\![x^3, x^4]\!]$ (resp. $R := k[x,y]$) is given where $k$ is any field and $x,y$ indeterminates over $k$, along with nonzero ideals $J \subseteq I$ of $R$ such that $\rr{J} \nsubseteq \rr{I}$.

Many of the topics and questions explored in this article could be applied to these not-quite-closure operations as well, but additional care is needed.

\section{Constructing closure operations}\label{sec:cons}
There are, however, some actions one can take which always produce closure operations.

\subsection{Standard constructions}\label{sub:cons}
\begin{construction}\label{con:module}
Let $U$ be an $R$-module.  Then the operation $I \mapsto I^{\rm cl} := \{f\in R \mid fU \subseteq IU\} = (IU :_R U)$ gives a closure operation on $R$.  Extension and order-preservation are clear.  As for idempotence, suppose $f\in (I^{\rm cl})^{\rm cl}$.  Then $fU \subseteq I^{\rm cl} U$.  But for any $g\in I^{\rm cl}$, $gU \subseteq IU$, whence $I^{\rm cl} U \subseteq IU$, so $fU \subseteq IU$ as required.

As we shall see, this is a very productive way to obtain closure operations, especially when $U$ is an $R$-algebra.  For example, letting $\ia$ be an ideal of $R$ and $U := R/\ia$, we see that the assignment $I \mapsto I+\ia$ gives a closure operation.  On the other hand, letting $U := \ia$, the resulting closure operation becomes $I \mapsto (I\ia : \ia)$, which is the basis for the \emph{$\Delta$-closures} of \cite{Rat-Delta} and for the \emph{basically full closure} of \cite{HRR-bf}.
\end{construction}

\begin{construction}\label{con:contraction}
We give here a variant on Construction~\ref{con:module}.

Let $\phi: R \ra S$ be a ring homomorphism and let $\rd$ be a closure operation on $S$.  For ideals $I$ of $R$, define $I^\rc := \phi^{-1}\left((\phi(I)S)^\rd\right)$.  (One might loosely write $I^\rc := (IS)^\rd \cap R$.)  Then $\rc$ is a closure operation on $R$.

Extension and order-preservation are clear.  As for idempotence, if $f\in \left(I^\rc\right)^\rc$, then $\phi(f) \in \left((I^\rc)S\right)^\rd \subseteq \left((IS)^\rd\right)^\rd = (IS)^\rd$, so that $f \in I^\rc$.
\end{construction}

\begin{construction}\label{con:intersection}
Let $\{\rm c_\lambda\}_{\lambda \in \Lambda}$ be an \emph{arbitrary collection} of closure operations on ideals of $R$.  Then $I^{\rm c} := \bigcap_{\lambda \in \Lambda} I^{\rm c_\lambda}$ gives a closure operation as well.\footnote{Similar considerations in the context of star-operations on integral domains are used in \cite{AnAn-exstar} to give lattice structures on certain classes of closure operations.}

Again, extension and order-preservation are clear.  As for idempotence, suppose $f \in (I^{\rm c})^{\rm c}$.  Then for every $\lambda \in \Lambda$, we have $f \in (I^{\rm c})^{\rm c_\lambda}$.  But since $I^{\rm c} \subseteq I^{\rm c_\lambda}$ and $\rm c_\lambda$ preserves order, we have \[
f \in (I^{\rm c})^{\rm c_\lambda} \subseteq (I^{\rm c_\lambda})^{\rm c_\lambda} = I^{\rm c_\lambda},
\]
where the last property follows from the idempotence of $\rm c_\lambda$.  Since $\lambda \in \Lambda$ was chosen arbitrarily, $f\in I^{\rm c}$ as required. 
\end{construction}

For the next construction, we need to mention the natural partial order on closure operations on a ring $R$.  Namely, if $\rm c$ and $\rm d$ are closure operations, we write $\rm c \leq \rm d$ if for every ideal $I$, $I^{\rm c} \subseteq I^{\rm d}$.

\begin{construction}\label{con:directunion}
Let $\{\rm c_\lambda\}_{\lambda \in \Lambda}$ be a \emph{directed set} of closure operations.  That is, for any $\lambda_1, \lambda_2 \in \Lambda$, there exists some $\mu \in \Lambda$ such that $\rc_{\lambda_i} \leq \rc_\mu$ for $i=1, 2$.  Moreover, assume that $R$ is Noetherian.  Then $I^{\rc} := \bigcup_{\lambda \in \Lambda} I^{\rc_\lambda}$ gives a closure operation.

First note that $I^\rc$ is indeed an ideal.  It is the sum $\sum_{\lambda \in \Lambda} I^{\rc_\lambda}$.  After this, extension and order-preservation are clear.  Next, we note that for any ideal $I$, there is some $\mu\in \Lambda$ such that $I^\rc = I^{\rc_\mu}$.  To see this, we use the fact that $I^\rc$ is finitely generated along with the directedness of the set $\{\rc_\lambda \mid \lambda \in \Lambda\}$.  Namely, $I^\rc = (f_1, \dotsc, f_n)$; each $f_i \in I^{\rc_{\lambda_i}}$; then let $\rc_\mu$ be such that $\rc_{\lambda_i} \leq \rc_\mu$ for $i=1, \dotsc, n$.

To show idempotence, take any ideal $I$.  By what we just showed, there exist $\lambda_1, \lambda_2 \in \Lambda$ such that $(I^\rc)^\rc = (I^\rc)^{\rc_{\lambda_1}}$ and $I^\rc = I^{\rc_{\lambda_2}}$.  Choose $\mu \in \Lambda$ such that $\rc_{\lambda_i} \leq \rc_\mu$ for $i=1, 2$.  Then \[
(I^\rc)^\rc = ( I^{\rc_{\lambda_2}})^{\rc_{\lambda_1}} \subseteq (I^{\rc_\mu})^{\rc_\mu} = I^{\rc_\mu} \subseteq I^\rc.
\]
\end{construction}

\begin{construction}\label{con:idemhull}
Let $\rd$ be an operation on (ideals of) $R$ that satisfies properties~(\ref{it:extend}) and (\ref{it:preserve}) of Definition~\ref{def:closure}, but is not idempotent.  Let $\cS$ be the set of all closure operations on $R$ defined by the property that $\rc \in \cS$ if and only if $I^\rd \subseteq I^\rc$ for all ideals $I$ of $R$.  Then by Construction~\ref{con:intersection}, the assignment $I \mapsto I^{\rd^\infty} := \bigcap_{\rc \in \cS} I^\rc$ is itself a closure operation, called the \emph{idempotent hull} of $\rd$ \cite[Section 4.6]{DikTho-closure}.  It is obviously the smallest closure operation lying above $\rd$.

If $R$ is Noetherian, it is equivalent to do the following: Let $\rd^1 := \rd$, and for each integer $n\geq 2$, we inductively define $\rd^n$ by setting $I^{\rd^n} := \left(I^{\rd^{n-1}}\right)^\rd$.  Let $I^{\rd'} := \bigcup_n I^{\rd^n}$ for all $I$.  One may routinely check that $\rd'$ is an extensive, order-preserving operation on ideals of $R$, and idempotence follows from the ascending chain condition on the ideals $\{I^{\rd^n}\}_{n \in \N}$. Clearly, $I^{\rd'} = I^{\rd^\infty}$.
\end{construction}

\begin{construction}\label{con:fintype}
This construction is only relevant when $R$ is not necessarily Noetherian.

Let $\rc$ be a closure operation.  Then we define $\rc_f$ by setting \[
I^{\rc_f} := \bigcup \{J^\rc \mid J \text{ a finitely generated ideal such that } J \subseteq I\}.
\]
This is a closure operation: Extension follows from looking at the principal ideals $(x)$ for all $x\in I$.  Order-preservation is obvious.  As for idempotence, suppose $z\in (I^{\rc_f})^{\rc_f}$.  Then there is some finitely generated ideal $J \subseteq I^{\rc_f}$ such that $z \in J^\rc$. Let $\{z_1, \dotsc, z_n\}$ be a finite generating set for $J$.  Since each $z_i \in I^{\rc_f}$, there exist finitely generated ideals $K_i \subseteq I$ such that $z_i \in K_i^\rc$.  Now let $K := \sum_{i=1}^n K_i$.  Then $J \subseteq K^\rc$, so that \[
z \in J^\rc \subseteq (K^\rc)^\rc = K^\rc,
\]
and since $K$ is a finitely generated sub-ideal of $I$, it follows that $z \in I^{\rc_f}$.
\end{construction}

If $\rc = \rc_f$, we say that $\rc$ is of \emph{finite type}.  Clearly $\rc_f$ is of finite type for any closure operation $\rc$, and it is the largest finite-type closure operation $\rd$ such that $\rd \leq \rc$.  Connected with this, we have the following:

\begin{prop}\label{pr:cfmax}
Let $\rc$ be a closure operation of finite type on $R$.  Then every $\rc$-closed ideal is contained in a $\rc$-closed ideal that is maximal among $\rc$-closed ideals.
\end{prop}

The proof is a standard Zorn's lemma argument.  The point is that the union of a chain of $\rc$-closed ideals is $\rc$-closed because $\rc$ is of finite type.

\subsection{Common closures as iterations of standard constructions}\label{sub:fromcon}
Here we will show that essentially all the closures we gave in Example~\ref{ex:closures} result as iterations of the constructions just given:

\begin{enumerate}
\item The \emph{identity} needs no particular construction.

\item The \emph{indiscrete closure} is an example of Construction~\ref{con:module}, by letting $U=0$.

\item As for the \emph{radical}, we use the characterization of it being the intersection of the prime ideals that contain the ideal.  Consider the maps $\pi_\p: R \ra R_\p / \p R_\p =: \kappa(\p)$ for prime ideals $\p$.  Note that \[
\pi_\p^{-1}(I \kappa(\p)) = \begin{cases}
\p, &\text{if } I \subseteq \p,\\
R, &\text{otherwise}.
\end{cases}
\]
Let $I^\p := \pi_\p^{-1}(I \kappa(\p))$.  This is an instance of Construction~\ref{con:module} with $U = \kappa(\p)$, and the intersection of all such closures (Construction~\ref{con:intersection}) is the radical.  That is, $\sqrt{I} = \bigcap_{\p \in \Spec R} I^\p$.
\item The \emph{$\ia$-saturation} may be obtained in one of two ways.  Assuming one already has the extensive, order-preserving operation $(-:\ia)$, then applying Construction~\ref{con:idemhull} yields the $\ia$-saturation.

Alternately, let $\{a_\lambda \mid \lambda \in \Lambda\}$ be a generating set for $\ia$, with each $a_\lambda \neq 0$.  Let $\ell_\lambda: R \ra R_{a_\lambda}$ be the localization map.  Then each $(I : a_\lambda^\infty) = \ell_\lambda^{-1}(I R_{a_\lambda}) =: I^\lambda$ is an instance of Construction~\ref{con:module} (or \ref{con:contraction}, if you like), and we may apply Construction~\ref{con:intersection} to get $(I : \ia^\infty) = \bigcap_{\lambda \in \Lambda} (I : a_\lambda^\infty) = \bigcap_{\lambda \in \Lambda} I^\lambda$.

\item For \emph{integral closure}, let $\p$ be a minimal prime of $R$, let $V$ be a valuation ring (or if $R/\p$ is Noetherian, it's enough to let $V$ be a rank 1 discrete valuation ring) between $R/\p$ and its fraction field, let $j_V: R \ra V$ be the natural map, and let $I^V := j_V^{-1}(IV)$ (which gives a closure operation via Construction~\ref{con:module} with $U=V$).  Then it is a theorem (e.g. \cite[Theorem 6.8.3]{HuSw-book}) that $I^- = \bigcap_{\text{all such } V} I^V$, which is an application of Construction~\ref{con:intersection}.

\item For \emph{plus closure}, \emph{when $R$ is an integral domain}, let $Q$ be its fraction field, $\overline{Q}$ an algebraic closure of $Q$, and let $R^+$ be the integral closure of $R$ in $\overline{Q}$.  That is, $R^+$ consists of all elements of $\overline{Q}$ that satisfy a \emph{monic} polynomial over $R$.  Then we let $I^+ := I R^+ \cap R$, by way of Construction~\ref{con:module} with $U=R^+$.

In the general case, where $R$ is not necessarily a domain, for each minimal prime $\p$ of $R$ we let $\pi_\p: R \ra R/\p$ be the natural surjection.  Then $I^+ := \bigcap_{\text{all such } \p} \pi_\p^{-1} ((I R/\p)^+)$, via Construction~\ref{con:contraction}.

\item\label{it:Frob} For \emph{Frobenius closure} (when $R$ has positive prime characteristic $p$), we introduce the left $R$-modules $^eR$ for all $e\in \N$.  ${}^eR$ has the same additive group structure as $R$ (with elements being denoted ${}^er$ for each $r\in R$), and $R$-module structure given as follows: For $a\in R$ and ${}^er \in {}^eR$, $a \cdot {}^er = {}^e{(a^{p^e}r)}$.  Let $f_e: R \ra {}^eR$ be the $R$-module map given by $a \mapsto a \cdot {}^e1 = {}^e(a^{p^e})$.  Let $F_e$ be the closure operation given by $I^{F_e} := f_e^{-1}(I \cdot {}^eR)$, via Construction~\ref{con:module}.  Note that this is a \emph{totally ordered set} (and hence a directed set) of closure operations, in that $F_e \leq F_{e+1}$ for all $e$, due to the $R$-module maps ${}^eR \ra {}^{e+1}R$ given by $^er \mapsto {}^{e+1}(r^p)$.  Thus, we may use Construction~\ref{con:directunion} to get $I^F := \bigcup_{e\in \N} I^{F_e}$.

\item For \emph{tight closure} (when $R$ has positive prime characteristic $p$), we cannot use these constructions directly.  However, recall the theorem \cite[Theorem 8.6]{Ho-solid} that under quite mild assumptions on $R$ (namely the same ones that guarantee persistence of tight closure, see \ref{sub:persist}), $I^* = I^\star$, and use the constructions for solid closure below.

\item For \emph{solid closure} (when $R$ is a complete local domain), letting $i_S: R \ra S$ for solid $R$-algebras $S$ and $I^S := i_S^{-1}(IS)$ by way of Construction~\ref{con:module} with $U=S$, we note that this is a directed set of closure operations, since \cite[Proposition 2.1a]{Ho-solid} if $S$ and $T$ are solid $R$-algebras, so is $S \otimes_R T$.  Thus, we have $I^\star = \bigcup_{\text{all such } i_S} I^S$ via Construction~\ref{con:directunion}.

For general $R$: Let $\m$ be a maximal ideal of $R$, $\widehat{R}^{\m}$ the completion of $R_\m$ at its maximal ideal, $\p$ a minimal prime of $\widehat{R}^\m$, and $u_{\m,\p}: R \ra \widehat{R}^\m/\p$ the natural map.  Then we use Constructions~\ref{con:intersection} and \ref{con:contraction} to get $I^\star := \bigcap_{\text{all such pairs } \m, \p} u_{\m,\p}^{-1}((I \widehat{R}^\m/\p)^\star)$.

\item For \emph{$\Delta$-closure}, first note that for any ideal $K \in \Delta$, $I^K := (IK : K)$ gives a closure operation via Construction~\ref{con:module} with $U=K$.  Next, note that the closure operations $\{(-)^K \mid K \in \Delta\}$ form a \emph{directed set}, since for any $H, K \in \Delta$, $I^H + I^K \subseteq I^{KH}$.  Thus, Construction~\ref{con:directunion} applies to give $I^\Delta := \bigcup_{K \in \Delta} I^K$.

\item For \emph{basically full closure}, we merely apply Construction~\ref{con:module} with $U=\m$.

\end{enumerate}

\section{Properties of closures}\label{sec:props}
\subsection{Star-, semi-prime, and prime operations}\label{sub:starsemi}
\begin{defn}
Let $\rcl$ be a closure operation for a ring $R$.  We say that $\rcl$ is \begin{enumerate}
\item \emph{semi-prime} \cite{Pet-asym} if for all ideals $I,J$ of $R$, we have $I \cdot J^\rcl \subseteq (IJ)^\rcl$.  (Equivalently, $\left(I^\rcl J^\rcl\right)^\rcl = (IJ)^\rcl$ for all $I, J$.)
\item a \emph{star-operation} \cite[chapter 32 and see below]{Gil-MIT} if for every ideal $J$ and every non-zerodivisor $x$ of $R$, $(x J)^\rcl = x \cdot (J^\rcl)$.
\item \emph{prime} \cite{Kr-idealbook, Kr-domains1, Rat-Delta} if it is a semi-prime star-operation.
\end{enumerate}
\end{defn}

\begin{sociocomment}
In the literature of so-called ``multiplicative ideal theory'' (which is, roughly, that branch of commutative algebra that uses \cite{Gil-MIT} as its basic textbook), the definition of \emph{star-operation} is somewhat different from the above.  Namely, one assumes first that $R$ is an integral domain, one defines star-operations on \emph{fractional ideals of $R$}.  However, when $R$ is a domain, it is equivalent to do as I have done above.  Moreover, the terminology of star-operations is different from the terminology of this article.  For instance, if $\rc$ is a star-operation on a domain $R$, then $I^\rc$ is not called the $\rc$-closure, but rather the \emph{$\rc$-envelope} (or sometimes \emph{$\rc$-image}) of $I$, and if $I=I^\rc$, then $I$ is said to be a \emph{$\rc$-ideal}.  For the sake of self-containedness, I have elected rather to use the terminology I was raised on.

The field of closure operations on Noetherian rings has remained nearly disjoint from the field of star- (and ``semistar-'') operations on integral domains.  I think this is largely because the two groups of people have historically been interested in very different problems and baseline assumptions.  Multiplicative ideal theorists do not like to assume their rings are Noetherian, for example.  But I feel it would save a good deal of energy if the two fields would come together to some extent.  After all, there are very reasonable assumptions under which tight, integral, plus, and Frobenius closures are prime- (and hence star-) operations (see below).  This provides the star-operation theorists with a fresh infusion of star-operations to study, and it provides those who study said closures with a fresh arsenal of tools with which to study them.

I take the point of view natural to one of my training, in which one generalizes from integral domains to general commutative rings.
\end{sociocomment}

First note the following:
\begin{lemma}\label{lem:semi}
Let $\rcl$ be a closure operation on an integral domain $R$. \begin{enumerate}
\item $\rcl$ is a semi-prime operation if and only if for all $x\in R$ and ideals $J \subseteq R$, we have $x \cdot J^\rcl \subseteq (xJ)^\rcl$.
\item If $R$ is an integral domain, then $\rcl$ is a star-operation if and only if it is prime.
\end{enumerate}
\end{lemma}

\begin{proof}
If $\rcl$ is a semi-prime operation, then for any $x\in R$ and ideal $J \subseteq R$, we have \[
x \cdot J^\rcl = (x) \cdot J^\rcl \subseteq ((x)J)^\rcl = (xJ)^\rcl.
\]
Conversely, suppose $x J^\rcl \subseteq (xJ)^\rcl$ for all $x$ and $J$.  Let $I$ be an ideal of $R$, and let $\{a_\lambda\}_{\lambda \in \Lambda}$ be a generating set for $I$.  Then \[
I \cdot J^\rcl = \sum_{\lambda \in \Lambda} a_\lambda \cdot J^\rcl \subseteq \sum_\lambda (a_\lambda J)^\rcl \subseteq \left(\sum_\lambda a_\lambda J\right)^\rcl = (I J)^\rcl,
\] so that $\rcl$ is semi-prime.

Now suppose $R$ is an integral domain.  By definition any prime operation must be a star-operation.  So let $\rcl$ be a star-operation on $R$.  To see that it is semi-prime, we use part (1).  For any $x\in R$, either $x=0$ or $x$ is a non-zerodivisor. Clearly $0 \cdot J^\rcl = 0 \subseteq 0^\rcl = (0J)^\rcl$.  And if $x$ is a non-zerodivisor, then $x \cdot J^\rcl = (xJ)^\rcl$ by definition of star-operation. 
\end{proof}

One reason why the star-operation property is useful is as follows: any star-operation admits a unique extension to the set of fractional ideals of $R$ (where a \emph{fractional ideal} is defined to be a submodule $M$ of $Q$, the total quotient ring of $R$, such that for some non-zerodivisor $f$ of $R$, $f M \subseteq R$).  Namely, if $\rcl$ is a star-operation and $M$ is a fractional ideal, an element $x\in Q$ is in $M^\rcl$ if $fx \in (fM)^\rcl$, where $f$ is a non-zerodivisor of $R$ such that $f M \subseteq R$.  After this observation, another important property of star-operations is that if two fractional ideals $M, N$ are isomorphic, their closures are isomorphic as well.

Star-operations are important in the study of so-called \emph{Kronecker function rings} (for a historical and topical overview of this connection, see \cite{FoLo-Kronecker}).  However, the star-operation property is somewhat limiting.  For instance, I leave it as an exercise for the reader to show that if $R$ is a local Noetherian ring, then the radical operation on $R$ is a star-operation if and only if $\depth R = 0$.  The only star-operation on a rank 1 discrete valuation ring is the identity.    On the other hand, it is well known that integral closure is a star-operation on $R$ if and only if $R$ is normal.  This is true of tight closure as well:

\begin{prop}
Consider the following property for a closure operation $\rcl$: \[
(\#): \text{For any non-zerodivisor $x\in R$ and any ideal $I$, } I^\rcl = ((x I)^\rcl : x).
\]
\begin{enumerate}
\item A closure operation $\rcl$ is a star-operation if and only if it satisfies $(\#)$ and $(x)^\rcl = (x)$ for all non-zerodivisors $x\in R$. 
\item Closure operations that satisfy \emph{(\#)} include plus-closure (when $R$ is a domain), integral closure, tight closure (in characteristic $p>0$), and Frobenius closure.
\end{enumerate}
\end{prop}

\begin{proof}
\begin{enumerate}
\item Suppose $\rcl$ is a star-operation, $x$ is a non-zerodivisor, and $I$ an ideal.  Then $(x)^\rcl = ((x)R)^\rcl = x \cdot R^\rcl = (x)$, and $((xI)^\rcl :x) = (x (I^\rcl) :x)$ (since $\rcl$ is a star-operation) $= I^\rcl$ (since $x$ is a non-zerodivisor).

Conversely, suppose $\rcl$ satisfies (\#) and that all principal ideals generated by non-zerodivisors are $\rcl$-closed.  For a non-zerodivisor $x$ and ideal $I$, we have $x \cdot I^\rcl = x \cdot ((xI)^\rcl : x) \subseteq (x I)^\rcl$, so we need only show that $(xI)^\rcl \subseteq x \cdot I^\rcl$.  So suppose $g \in (x I)^\rcl$.  Since $x I \subseteq (x)$, it follows that $g \in (x)^\rcl = (x)$, so $g =xf$ for some $f\in R$.  Thus, $xf \in (x I)^\rcl$, so $f \in ((x I)^\rcl :x) = I^\rcl$, whence $g = xf \in x \cdot I^\rcl$ as required.

\item (Frobenius closure): Let $g \in ((x I)^F :x)$.  Then $xg \in (xI)^F$, so there is some $q=p^n$ such that \[
x^q g^q = (xg)^q \in (xI)^{[q]} = x^q I^{[q]}.
\]
Since $x^q$ is a non-zerodivisor, $g^q \in I^{[q]}$, whence $g \in I^F$.

(Tight closure): The proof is similar to the Frobenius closure case.

(Plus closure): If $xg \in (x I)^+$, then there is some module-finite domain extension $R \subseteq S$ such that $xg \in x IS$.  But since $x$ is a non-zero element of the domain $S$ (hence a non-zerodivisor on $S$), it follows that $g \in IS$, whence $g \in I^+$.

(Integral closure): Suppose $xg \in (x I)^-$.  Then there is some $n\in \N$ and elements $a_i \in (xI)^i$ ($1\leq i \leq n$) such that \[
(xg)^n + \sum_{i=1}^n a_i (xg)^{n-i}=0.
\]
But each $a_i \in (xI)^i = x^i I^i$, so for some $b_i \in I^i$ (for each $i$), we have $a_i = x^i b_i$.  Then the displayed equation yields: \[
x^n \left(g^n + \sum_{i=1}^n b_i g^{n-i}\right) = 0,
\]
and since $x^n$ is a non-zerodivisor, it follows that $g \in I^-$.
\end{enumerate}
\end{proof}

Semi-prime operations, however, are ubiquitous.  (In fact, some authors \cite{Kir-closure} even include the property in their basic definition of what a closure operation is!)  One can, of course, cook up a non-semi-prime closure operation, even on a rank 1 discrete valuation ring \cite[Example 2.3]{Va-cl}.  However, \emph{essentially all the examples and constructions explored so far yield semi-prime operations}, in the following sense (noting that all of the following statements have easy proofs) : \begin{itemize}
\item Any closure arising from Construction~\ref{con:module} is semi-prime.
\item In Construction~\ref{con:contraction}, if $\rd$ is a semi-prime operation on $S$, then $\rc$ is a semi-prime operation on $R$.
\item In Construction~\ref{con:intersection}, if every $\rc_\lambda$ is semi-prime, then so is $\rc$.
\item In Construction~\ref{con:directunion}, if every $\rc_\lambda$ is semi-prime, then so is $\rc$.
\item In Construction~\ref{con:idemhull}, if $I \cdot J^\rd \subseteq (IJ)^\rd$ for all ideals $I, J$ of $R$, then $\rd^\infty$ is semi-prime.
\item In Construction~\ref{con:fintype}, if $\rc$ is semi-prime, then so is $\rc_f$.
\item Hence by \ref{sub:fromcon}, all of the closures from Example~\ref{ex:closures} are semi-prime.\footnote{One need not go through solid closure to show that tight closure must always be semi-prime.}
\end{itemize}

Here are some nice properties of semi-prime closure operations:

\begin{prop}\label{pr:prdec}
Let $\rcl$ be a semi-prime closure operation on $R$.  Let $I$, $J$ be ideals of $R$, and $W$ a multiplicatively closed subset of $R$. \begin{enumerate}
\item $(I:J)^\rcl \subseteq (I^\rcl :J)$.  Hence if $I$ is $\rcl$-closed, then so is $(I : J)$.
\item $(I^\rcl : J)$ is $\rcl$-closed.
\item If $R$ is Noetherian and $I$ is $\rcl$-closed, then $(I W^{-1}R) \cap R$ is $\rcl$-closed.
\item If $R$ is Noetherian and $I$ is $\rcl$-closed, then all the \emph{minimal primary components} of $I$ are $\rcl$-closed.  Hence, if $I=I^\rcl$ has no embedded components, it has a \emph{primary decomposition by $\rcl$-closed ideals}.
\item The maximal elements of the set $\{I \mid I^\rcl = I \neq R\}$ are prime ideals.
\end{enumerate}
\end{prop}

\begin{proof}\
\begin{enumerate}
\item Let $f\in (I :J)^\rcl$.  Then $Jf \subseteq J (I :J)^\rcl \subseteq (J \cdot (I :J))^\rcl \subseteq I^\rcl.$
\item follows directly from part (1).
\item Let $J := (I W^{-1}R) \cap R$.  Then $J = \{f \in R \mid \exists w\in W \text{ such that } wf \in I\}$.  But $J$ is finitely generated (since $R$ is Noetherian); say $J = (f_1, \dotsc, f_n)$.  Then for each $1\leq i \leq n$, there exists $w_i \in W$ such that $w_i f_i \in I$.  Let $w := \prod_{i=1}^n w_i$.  Then $wJ \subseteq I$, so $J \subseteq (I : w)$.  But it is obvious that $(I : w) \subseteq J$, so $J = (I : w)$.  Then the conclusion follows from part (1).
\item The minimal primary components of $I$ look like $I R_P \cap R$, for each minimal prime $P$ over $I$.  Then the conclusion follows from part (3).
\item Let $I$ be such a maximal element.  Let $x,y \in R$ such that $xy \in I$ and $y\notin I$.  Then $(I :x)$ is a $\rcl$-closed ideal (by part (1)) that properly contains $I$ (since $y \in (I :x) \setminus I$), so since $I$ is maximal among proper $\rcl$-closed ideals, it follows that $(I:x) = R$, which means that $x\in I$.
\end{enumerate}
\end{proof}

Finally, here is a construction on semi-prime operations:

\begin{construction}\label{con:w}
Let $\rc$ be a semi-prime closure operation on $R$.  Let $\rc_f$-$\Max R$ (see Construction~\ref{con:fintype} for the definition of $\rc_f$) denote the set of $\rc_f$-closed ideals which are maximal among the set of all $\rc_f$-closed ideals.  By Proposition~\ref{pr:prdec}(5), $\rc_f$-$\Max R$ consists of prime ideals, and by Proposition~\ref{pr:cfmax}, every $\rc_f$-closed ideal is contained in a member of $\rc_f$-$\Max R$.  Then we define ${\rc_w}$ as follows: 
\[
I^{{\rc_w}} := \{x \in R \mid \forall \p \in \rc_f\text{-}\Max R,\ \exists d \in R \setminus \p \text{ such that } dx \in I\}.
\]
In other words, $I^{{\rc_w}}$ consists of all the elements of $R$ that land in the extension of $I$ to all localizations $R \ra R_\p$ for $\p \in \rc_f$-$\Max R$.  As this arises from Constructions~\ref{con:module} and \ref{con:intersection}, ${\rc_w}$ is a semi-prime closure operation.

Moreover, ${\rc_w} \leq \rc_f$.  To see this, let $x\in I^{{\rc_w}}$.  Then for all $\p \in \rc_f$-$\Max R$, there exists $d_\p \in R \setminus \p$ with $d_\p x \in I$.  Let $J$ be the ideal generated by the set $\{d_\p \mid \p \in \rc_f$-$\Max R\}$.  Then $Jx \subseteq I$ and $J^{\rc_f} = R$, so since $\rc_f$ is semi-prime, we have \[
x = 1\cdot x \in R (x) = J^{\rc_f} (x) \subseteq (Jx)^{\rc_f} \subseteq I^{\rc_f}.
\]
\end{construction}

If $R$ is a domain, and $\rc$ is a star-operation (i.e. prime), then this construction is essentially due to \cite{AnCo-twostar}, who show that in this context $\rc_w$ distributes over finite intersections, is of finite type, and is the \emph{largest} star-operation $\rd$ of finite type that distributes over finite intersection such that $\rd \leq \rc$.

\subsubsection{The $\rv$-operation}
Arguably the most important star-operation (at least in the theory of star-operations \emph{per se}) is the so-called $\rv$-operation.  Classically it was only defined when $R$ is a domain \cite[chapters 16, 32, 34]{Gil-MIT}, but it works in general.  Most star-operations in the literature (among those that are identified as star-operations) are based in one way or another on the $\rv$-operation:

\begin{defn}
Let $R$ be a ring and $Q$ its total quotient ring.  For an ideal $I$, the set $I_\rv$ is defined to be the intersection of all cyclic $R$-submodules $M$ of $Q$ such that $I \subseteq M$.
\end{defn}

\begin{prop}\label{pr:vop}\
\begin{enumerate}
\item $\rv$ is a star-operation.
\item For any star-operation $\rcl$ on $R$, $I^\rcl \subseteq I_\rv$ for all ideals $I$ of $R$.  (That is, $\rv$ is the \emph{largest} star-operation on $R$.)
\item There exists a ring $R$ for which $\rv$ is \emph{not} semi-prime (and hence not prime).
\end{enumerate}
\end{prop}

\begin{proof}
It is easy to see that $\rv$ is a closure operation.  For the star-operation property, let $x$ be a non-zerodivisor and $I$ an ideal of $R$. Let $a\in I_\rv$.  Let $M$ be a cyclic submodule of $Q$ that contains $xI$.  Then $M = R \cdot \frac{r}{s}$ for some $r, s\in R$ with $s$ a non-zerodivisor.  Moreover, $xI \subseteq M$ implies that $I \subseteq R \cdot \frac{r}{sx}$.  Since $a\in I_\rv$, we have $a\in R \cdot \frac{r}{sx}$ as well, so that $xa \in R \cdot \frac{r}{s}=M$. Thus, $xa \in (xI)_\rv$ as required.

For the opposite inclusion, we first note that principal ideals are $\rv$-closed because they are cyclic $R$-submodules of $Q$.  Now let $a \in (xI)_\rv$.  Since $xI \subseteq (x)$, we have $a \in (xI)_\rv \subseteq (x)_\rv = (x)$, so $a=xb$ for some $b\in R$.  Let $M = R \cdot \frac{r}{s}$ be a cyclic submodule of $Q$ that contains $I$.  Then $xI \subseteq xM = R \cdot \frac{xr}{s}$, so since $a \in (xI)_\rv$, it follows that $xb = a \in xM = R \cdot \frac{xr}{s} = x \cdot R \cdot \frac{r}{s}$.  Since $x$ is a non-zerodivisor on $Q$, we can cancel it to get $b \in R \cdot \frac{r}{s} = M$, whence $b \in I_\rv$.  Thus, $a =xb \in x \cdot I_\rv$, as required.

Now let $\rcl$ be an arbitrary star-operation on $R$, and $I$ an ideal.  Let $M = R \cdot \frac{r}{s}$ be a cyclic submodule of $R$ that contains $I$ (so that $r,s \in R$ and $s$ is not a zerodivisor).  Then $sI \subseteq rR = (r)$, so that \[
s \cdot I^\rcl = (s I)^\rcl \subseteq (r)^\rcl = (r).
\]
That is, $I^\rcl \subseteq R \cdot \frac{r}{s} = M$.  Thus, $I^\rcl \subseteq I_\rv$.

For the counterexample, let \[
R := k[X,Y] / (X^2, XY, Y^2) = k[x,y],
\] where $k$ is a field and $X,Y$ are indeterminates over $k$ (with the images in $R$ denoted $x,y$ respectively).  Let $\m := (x,y)$ be the unique maximal ideal of $R$.  Since $R$ is an Artinian local ring, it is equal to its own total ring of quotients, and so the cyclic $R$-submodules of said ring of quotients are just the principal ideals of $R$.  Since $\m$ is not contained in any proper principal ideal of $R$, we have $\m_\rv = R$.  Thus, $\m \cdot \m_\rv = \m \cdot R = \m$.  On the other hand, $(\m \m)_\rv = (\m^2)_\rv = 0_\rv = 0$, which shows that $\m \cdot \m_\rv \nsubseteq (\m \m)_\rv$, whence $\rv$ is not semi-prime.
\end{proof}

\begin{rmk}
When $R$ is an integral domain, parts (1) and (2) of the above Proposition are well-known (and since by Lemma~\ref{lem:semi}, any star-operation on a domain is prime, the analogue of (3) is false).  Two other well-known properties of the $\rv$-operation in the domain case are as follows: \begin{itemize}
\item $\Hom_R(\Hom_R(I,R), R) \cong I_\rv$ as $R$-modules.   For this reason, the $\rv$-operation is sometimes also called the \emph{reflexive hull} operation on ideals.
\item $I_\rv = (I^{-1})^{-1}$.  (Recall that for a fractional ideal $J$ of $R$, $J^{-1} := \{x \in Q \mid xJ \subseteq R\}$, where $Q$ is the quotient field of $R$.)  For this reason, the $\rv$-operation is sometimes also called the \emph{divisorial closure}.
\end{itemize}
These ideas have obvious connections to Picard groups and divisor class groups.
\end{rmk}

The \emph{$\rt$-} and {$\rw$-operations} (see e.g. \cite{Za-tinv}) should be mentioned here as well.  By definition, $\rt := \rv_f$ (via Construction~\ref{con:fintype}).   When $R$ is a domain, the $\rw$-operation is defined by $\rw := \rv_w$ (by Construction~\ref{con:w}). So if $R$ is a domain, then $\rt$ is semi-prime, but otherwise it need not be (as the counterexample in Proposition~\ref{pr:vop} shows), and $\rw$ may not even be well-defined in the non-domain case.

\subsection{Closures defined by properties of (generic) forcing algebras}\label{sub:forcing}
Let $R$ be a ring, $I$ a (finitely generated) ideal and $f\in R$.  Then a \emph{forcing algebra} \cite{Ho-solid} for $[I; f]$ consists of an $R$-algebra $A$ such that (the image of) $f \in IA$.  In particular, given a generating set $I = (f_1, \dotsc, f_n)$, one may construct the \emph{generic forcing algebra} $A$ for the data $[f_1, \dotsc, f_n; f]$, given by \[
A := R[T_1, \dotsc, T_n] / (f + \sum_{i=1}^n f_i T_i).
\] Clearly $A$ is a forcing algebra for $[I; f]$.  Moreover, if $B$ is any other forcing algebra for $[I;f]$, there is an $R$-algebra map $A \ra B$.  To see this, if $f\in IB$, then there exist $b_1, \dotsc, b_n \in B$ such that $f + \sum_{i=1}^n f_i b_i=0$.  Then we define the map $A \ra B$ by sending each $T_i \mapsto b_i$.

Many closure operations may be characterized by properties of generic forcing algebras.  This viewpoint is explored in some detail in \cite{Br-Groth}, where connections are also made with so-called \emph{Grothendieck topologies}.  We list a few (taken from \cite{Br-Groth}), letting $f$, $I:= (f_1, \dotsc, f_n)$, and $A$ be as above: \begin{itemize}
\item $f \in I$ (the identity closure of $I$) if and only if $R$ is a forcing algebra for $[f;I]$.  That is, $f\in I$ if and only if there is an $R$-algebra map $A \ra R$.  In geometric terms, one says that the structure map $\Spec A \ra \Spec R$ has a \emph{section}.
\item $f\in \sqrt{I}$ if and only if $f \in IK$ for all fields $K$ that are $R$-algebras, if and only if the ring map $R \ra A$ has the \emph{lying-over} property, if and only if the map $\Spec A \ra \Spec R$ is \emph{surjective} (as a set map).
\item $f \in I^-$ if and only if $f \in IV$ for all rank-1 discrete valuations of $R$ (appropriately defined), if and only if for all such $V$, there is an $R$-algebra map $A \ra V$.  It is not immediately obvious, but this is equivalent to the topological property that the map $\Spec A \ra \Spec R$ is \emph{universally submersive} (also called a \emph{universal topological epimorphism}).
\item If $R$ has prime characteristic $p>0$, $f \in I^F$ if and only if $f \in I R_\infty$, if and only if there is an $R$-algebra map $A \ra R_\infty$.
\end{itemize}
The final closure to note in this context is solid closure (the connection of which to tight closure has already been noted).  For simplicity, let $(R,\m)$ be a complete local domain of dimension $d$.  We have $f \in I^\star$ if and only if there is some solid $R$-algebra $S$ such that $f \in IS$, i.e. iff there is an $R$-algebra map $A \ra S$.  Hochster \cite[Corollary 2.4]{Ho-solid} showed in turn that this is equivalent to the condition that $H^d_\m(A) \neq 0$. (!)  This viewpoint brings in all sorts of cohomological tools into the study of solid closure, and hence tight closure in characteristic $p$.  Such tools were crucial in the proof that tight closure does not always commute with localization \cite{BM-unloc}.

\subsection{Persistence}\label{sub:persist}
Although it is possible to do so, usually one does not define a closure operation one ring at a time.  The more common thing to do is define the closure operation for a whole class of rings.  In such cases, the most important closure operations are \emph{persistent}:

\begin{defn}
Let $\cR$ be a full subcategory of the category of commutative rings; let $\rc$ be a closure operation defined on the rings of $\cR$.  We say that $\rc$ is \emph{persistent} if for any ring homomorphism $\phi: R \ra S$ in $\cR$ and any ideal $I$ of $R$, one has $\phi(I^\rc)S \subseteq (\phi(I)S)^\rc$.
\end{defn}

Common choices for $\cR$ are \begin{itemize}
\item All rings and ring homomorphisms.
\item Any full subcategory of the category of rings.
\item Graded rings and graded homomorphisms.
\item Local rings and local homomorphisms.
\end{itemize}

For instance, it is easy to show that radical and integral closure are persistent on the category of all rings (as are the identity and indiscrete closures), and that Frobenius closure is persistent on characteristic $p$ rings.  Tight closure is persistent along maps $R \ra S$ of characteristic $p$ rings, as long as either $R/\sqrt{0}$ is $F$-finite or $R$ is essentially of finite type over an excellent local ring, although this is truly a deep theorem \cite{HHbase}.  On the other hand, tight closure is persistent on the category of equal characteristic 0 rings because of the way it is defined (see the discussion after Theorem~\ref{thm:tcgood}).  Saturation (with respect to the maximal ideal) is persistent on the category of local rings and local homomorphisms, as well as on the category of graded rings and graded homomorphisms over a fixed base field.

Plus closure is also persistent on the category of integral domains, as is evident from the fact that the operation of taking absolute integral closure of the domains involved is \emph{weakly functorial}, in the sense that any such map $R \ra S$ extends (not necessarily uniquely) to a map $R^+ \ra S^+$ \cite[p. 139]{HH-bull}.  This argument may be extended to show that plus closure is persistent on the category of \emph{all} rings as well, by considering minimal primes.

Basically full closure, however, is not persistent, even if we restrict to complete local rings of dimension one, $\m$-primary ideals, and local homomorphisms $R \ra S$ such that $S$ becomes a finite $R$-module.  For a counterexample, let $k$ be any field, let $x, y, z$ be analytic indeterminates over $k$, let $R := k[\![x,y]\!] / (x^2, xy)$, $I := yR$, $S := k[\![x,y,z]\!] / (x^2, xy, z^2)$, and let the map $(R,\m) \ra (S,\n)$ be the obvious inclusion.  Then $I^\bfc = (\m y :_R \m) = (x,y)$ because $x$ is killed by all of $\m$.  But $x \notin (IS)^\bfc = (\n y :_S \n)$ because $zx \in \n x \setminus \n y$.  (Indeed, $IS$ is a basically full ideal of $S$.)

\subsection{Axioms related to the homological conjectures}\label{sub:homax}
A treatment of closure operations would not be complete without mentioning the so-called ``homological conjectures'' (for which we assume all rings are Noetherian).  So named by Mel Hochster, these comprise a complex list of reasonable-sounding statements that have been central to research in commutative algebra since the 1970s.  For the original treatment, see \cite{HoCBMS}. For a more modern treatment, see \cite{Ho-state}.  Rather than trying to cover the topic comprehensively, consider the following two conjectures:

\begin{conj}[Direct Summand Conjecture] \label{conj:dsc} Let $R \rightarrow S$ be an injective ring homomorphism, where $R$ is a regular local ring, such that $S$ is module-finite over $R$.  Then $R$ is a direct summand of $S$, considered as $R$-modules.
\end{conj}

\begin{conj}[Cohen-Macaulayness of Direct Summands conjecture] \label{conj:CM} Let $A \rightarrow R$ be a ring homomorphism which makes $A$ a direct summand of $R$, and suppose $R$ is regular.  Then $A$ is Cohen-Macaulay.
\end{conj}

Assuming a ``sufficiently good'' closure operation in mixed characteristic, these conjectures would be theorems.  Indeed, they \emph{are} theorems in equal characteristic, a fact which can be seen as a consequence of the existence of tight closure, as discussed in \ref{sub:tcim} (although the proof in equal characteristic of Conjecture~\ref{conj:dsc} predated tight closure by some 15 years \cite{Ho-ctrc}, as did a proof of some special cases of Conjecture~\ref{conj:CM} that come from invariant theory \cite{HoRo-invar}).

Consider the following axioms for a closure operation $\rcl$ on a category $\cR$ of Noetherian rings: \begin{enumerate}
\item (Persistence) For ring maps $R \rightarrow S$ in $\cR$, we have $I^\rcl S \subseteq (I S)^\rcl$.
\item (Tightness) If $R$ is a regular ring in $\cR$, then $I^\rcl = I$ for all ideals $I$ of $R$.
\item (Plus-capturing) If $R \rightarrow S$ is a module-finite extensions of integral domains in $\cR$ and $I$ is any ideal of $R$, then $I S \cap R \subseteq I^\rcl$. (i.e. $I^+ \subseteq I^\rcl$)
\item (Colon-capturing) Let $R$ be a \emph{local} ring in $\cR$ and $x_1, \dotsc, x_d$ a system of parameters.  Then for all $0 \leq i \leq d-1$, $(x_1, \dotsc, x_i) : x_{i+1} \subseteq (x_1, \dotsc, x_i)^\rcl$.
\end{enumerate}

\noindent \textbf{Notation:} Let $d\in \N$, and let $a,b$ be numbers such that either $a=b$ is a rational prime number, $a=b=0$, or $a=0$ and $b$ is a rational prime number.  For any such triple $(d,a,b)$, let $\cR_{d,a,b}$ be the category of complete local domains $(R,\m)$ such that $\dim R=d$, $\chr R=a$, and $\chr (R/\m) = b$.

\begin{prop}\label{pr:dsc}
Let $(d,a,b)$ be a triple as above.

Suppose a closure operation $\rcl$ exists on the category $\cR_{d,a,b}$ that satisfies conditions (1), (2), and (3) above.  Then the Direct Summand Conjecture holds in the category $\cR_{d,a,b}$.
\end{prop}

\begin{proof}
Let $R \rightarrow S$ be a module-finite injective ring homomorphism in $\cR_{d,a,b}$, where $R$ is a regular local ring.  Let $I$ be any ideal of $R$.  Then $I \subseteq I S \cap R \subseteq I^\rcl = I$ by properties (3) and (2).  That is, $R \ra S$ is \emph{cyclically pure}.  But then by \cite{Ho-purity}, since $R$ is a Noetherian domain it follows that $R \ra S$ is \emph{pure}.  Since $S$ is module-finite over $R$, it follows that $R$ is a direct summand of $S$ \cite[Corollary 5.3]{HoRo-purity}.
\end{proof}

\begin{prop}\label{pr:cmc}
Let $(d,a,b)$ be as above.

Suppose a closure operation $\rcl$ exists on the category $\cR_{d,a,b}$ that satisfies conditions (1), (2), and (4) above.   Then Conjecture~\ref{conj:CM} holds in the category $\cR_{d,a,b}$.
\end{prop}

\begin{proof}
Let $A \ra R$ be a local ring homomorphism of Noetherian local rings such that $A$ is a direct summand of $R$, and suppose $R$ is regular.  Let $x_1, \dotsc, x_d$ be a system of parameters for $A$, and pick some $0\leq i <d$.  Then \begin{align*}
(x_1, \dotsc, x_i) : x_{i+1} &\subseteq (x_1, \dotsc, x_i)^\rcl \subseteq (x_1, \dotsc, x_i)^\rcl R \cap A\\
&\subseteq ((x_1,\dotsc, x_i)R)^\rcl \cap A = (x_1, \dotsc, x_i)R \cap A = (x_1, \dotsc, x_i).
\end{align*}
Thus, $A$ is Cohen-Macaulay.
\end{proof}

\begin{rmk}
These ideas are closely related to Hochster's \emph{big Cohen-Macaulay modules} conjecture.  Indeed, Conjecture~\ref{conj:dsc} holds when $R$ has a so-called `big Cohen-Macaulay module' (see e.g. \cite[the final remark of \S4]{Ho-state}).  When $R$ is a complete local domain, Dietz \cite{Di-clCM} has given axioms for a persistent, residual closure operation on $R$-\emph{modules} (see \S\ref{sec:modules} for the basics on module closures) that are equivalent to the existence of a big Cohen-Macaulay module over $R$.
\end{rmk}

\subsection{Tight closure and its imitators}\label{sub:tcim}

Tight closure has been used, among other things, to carry out the program laid out in \ref{sub:homax} in cases where the ring contains a field.  Indeed, we have the following:
\begin{thm}\label{thm:tcgood}\cite{HHmain, HH-tcz}
Consider the category $\cR := \cR_{d,p,p}$, where $d$ is any nonnegative integer and $p\geq 0$ is either a prime number or zero.  Then we have \begin{enumerate}
\item  (Persistence) For ring maps $R \rightarrow S$ in $\cR$, we have $I^* S \subseteq (I S)^*$.
\item (Tightness) If $R$ is a regular ring in $\cR$, then $I^* = I$ for all ideals $I$ of $R$.
\item (Plus-capturing) If $R \rightarrow S$ is a module-finite extensions of integral domains in $\cR$ and $I$ is any ideal of $R$, then $I S \cap R \subseteq I^*$. (i.e. $I^+ \subseteq I^*$)
\item (Colon-capturing) Let $R$ be a \emph{local} ring in $\cR$ and $x_1, \dotsc, x_d$ a system of parameters.  Then for all $0 \leq i \leq d-1$, $(x_1, \dotsc, x_i) : x_{i+1} \subseteq (x_1, \dotsc, x_i)^*$.
\item (Brian\c{c}on-Skoda property) For any $R \in \cR$ and any ideal $I$ of $R$, \[
(I^d)^- \subseteq I^* \subseteq I^-.
\]
\end{enumerate}
Hence, by Propositions~\ref{pr:dsc} and \ref{pr:cmc}, both the Direct Summand Conjecture and Conjecture~\ref{conj:CM} hold in equal characteristic.
\end{thm}

Tight closure is defined (in \cite{HH-tcz}) for finitely generated $\Q$-algebras by a process of ``reduction to characteristic $p$'', a time-honored technique that we will not get into here (but see Definition~\ref{def:type} for a baby version of it).  Next, tight closure was defined on arbitrary (excellent) $\Q$-algebras in an ``equational'' way (see \ref{sub:heq}), and then Artin approximation must be employed to demonstrate that it has the properties given in Theorem~\ref{thm:tcgood}.  This is a long process, so various attempts have been made to give a closure operation in equal characteristic 0 that circumvents it.  More importantly, though, people have been trying to obtain a closure operation in \emph{mixed} characteristic that has the right properties.

The first such attempt was probably solid closure  \cite{Ho-solid}, already discussed.  Unfortunately, it fails tightness for regular rings of dimension 3  \cite{Ro-solid}.   Parasolid closure \cite{Br-rescue} (a variant of solid closure) agrees with tight closure in characteristic $p$, and it has all the right properties in equal characteristic 0, but is not necessarily easier to work with than tight closure, and it may or may not have the right properties in mixed characteristic.  Other at least partially successful attempts include parameter tight closure \cite{Ho-param}, diamond closure \cite{HoVe-diamond} and dagger closure (defined in \cite{HH-dagger}, but shown to satisfy tightness just recently in \cite{BrSt-dagreg}).

The most successful progress on the homological conjectures since the advent of tight closure theory is probably represented by Heitmann's proof \cite{Heit-ds3} of the direct summand conjecture (Conjecture~\ref{conj:dsc}) for $\cR_{3,0,p}$ (i.e. in mixed characteristic in dimension 3), which he does according to the program laid out above. Indeed, he shows the analogue of Theorem~\ref{thm:tcgood} when $\cR=\cR_{3,0,p}$ and tight closure is replaced everywhere with \emph{extended plus closure}, denoted $\repf$, first defined in his earlier paper \cite{Heit-ex+}.  (So in fact, his proof works to show the dimension 3 version of Conjecture~\ref{conj:CM} as well.)

\subsection{(Homogeneous) equational closures and localization}\label{sub:heq}

Let $\rcl$ be a closure operation on some category $\cR$ of rings.  We say that $\rcl$ \emph{commutes with localization} in $\cR$ if for any ring $R$ and any multiplicative set $W \subseteq R$ such that the localization map $R \ra W^{-1}R$ is in $\cR$, and for any ideal $I$ of $R$, we have $I^\rcl (W^{-1}R) = (I W^{-1}R)^\rcl$.

Tight closure does not commute with localization \cite{BM-unloc}, unlike Frobenius closure, plus closure, integral closure, and radical.  In joint work with Mel Hochster \cite{nmeHo-heq}, we investigated what it is that makes a persistent closure operation commute with localization, and in so doing we construct a tight-closure-like operation that \emph{does} commute with localization.  Our closure is in general smaller than tight closure, though in many cases (see below), it does in fact coincide with tight closure.

To do this, we introduce the concepts of \emph{equational} and \emph{homogeneous(ly equational)} closure operations.  The former notion is given implicitly in \cite{HH-tcz}, where Hochster and Huneke define the characteristic 0 version of tight closure in an equational way.

In the following, we let $\Lambda$ be a fixed base ring.  Often $\Lambda = \F_p$, $\Z$, or $\Q$.

\begin{defn}
Let $\cF$ be the category of \emph{finitely generated} $\Lambda$-algebras, and let $\rc$ be a persistent closure operation on $\cF$.  Then the \emph{equational} version of $\rc$, denoted $\rc \eq$, is defined on the category of \emph{all} $\Lambda$-algebras as follows:  Let $R$ be a $\Lambda$-algebra, $f\in R$, and $I$ an ideal of $R$.  Then $f\in I^{\rc \eq}$ if there exists $A \in \cF$, an ideal $J$ of $A$, $g\in A$, and a $\Lambda$-algebra map $\phi: A \ra R$ such that $g\in J^\rc$, $\phi(g)=f$, and $\phi(J) \subseteq I$.

Let $\cG$ be the category of finitely generated \emph{$\N$-graded} $\Lambda$-algebras $A$ which have the property that $[A]_0 = \Lambda$.  Let $\rc$ be a persistent closure operation on $\cG$.  Then the \emph{homogeneous(ly equational)} version of $\rc$, denoted $\rc \h$, is defined on the category of \emph{all} $\Lambda$-algebras as follows: Let $R$ be a $\Lambda$-algebra, $f\in R$, and $I$ an ideal of $R$.  Then $f \in I^{\rc \h}$ if there exists $A \in \cG$, a \emph{homogeneous} ideal $J$ of $A$, a \emph{homogeneous} element $g\in A$, and a $\Lambda$-algebra map $\phi: A \ra R$ such that $g \in J^\rc$, $\phi(g)=f$, and $\phi(J) \subseteq I$.

If $\rc$ is a closure operation on $\Lambda$-algebras, we say it is \emph{equational} if one always has $I^\rc = I^{\rc \eq}$, or \emph{homogeneous} if one always has $I^\rc = I^{\rc \h}$.
\end{defn}

In \cite{nmeHo-heq}, we prove the following theorem:

\begin{thm}
Suppose $\rc$ is a homogeneous closure operation.  Then it commutes with arbitrary localization.  That is, if $R$ is a $\Lambda$-algebra, $I$ an ideal, and $W$ a multiplicative subset of $R$, then $(W^{-1}I)^\rc = W^{-1}(I^\rc)$.
\end{thm}

In particular, \emph{homogeneous tight closure} ($I^{*\h}$) commutes with localization. The following theorem gives circumstances under which $I^* = I^{*\h}$.  Parts (1) and (3) involve cases where tight closure was already known to commute with localization, thus providing a ``reason'' that it commutes in these cases.  Part (2) gives a reason why the coefficient field in Brenner and Monsky's counterexample is transcendental over $\F_2$.
\begin{thm}
Let $R$ be an excellent Noetherian ring which is either of prime characteristic $p>0$ or of equal characteristic $0$. \begin{enumerate}
\item If $I$ is a parameter ideal (or more generally, if $R/I$ has \emph{finite phantom projective dimension} as an $R$-module), then $I^* = I^{*\h}$.
\item If $R$ is a finitely generated and positively-graded $k$-algebra, where $k$ is an algebraic extension of $\F_p$ or of $\Q$, and $I$ is an ideal generated by forms of positive degree, then $I^* = I^{*\h}$.
\item If $R$ is a \emph{binomial ring} over any field $k$ (that is, $R = k[X_1, \dotsc, X_n] / J$, where the $X_j$ are indeterminates and $J$ is generated by polynomials with at most two terms each), then for any ideal $I$ of $R$, $I^* = I^{*\h}$.
\end{enumerate}
\end{thm}

\section{Reductions, special parts of closures, spreads, and cores}\label{sec:red}
\subsection{Nakayama closures and reductions}
\begin{defn}\cite{nme*spread}
Let $(R,\m)$ be a Noetherian local ring.  A closure operation $\rc$ on ideals of $R$ is \emph{Nakayama}  if whenever $J$, $I$ are ideals such that $J \subseteq   I \subseteq (J+\m I)^\rc$, it follows that $J^\rc = I^\rc$.
\end{defn}

It turns out that many closure operations are Nakayama.  For example, integral closure \cite{NR},  tight closure \cite{nme*spread}, plus closure \cite{nmePhD}, and Frobenius closure \cite{nme-sp} are all Nakayama closures under extremely mild conditions.  The fact that the identity closure is Nakayama is a special case of the classical Nakayama lemma (which is where the name of the condition comes from).  However, radical is not Nakayama.  For example, if $R=k[\![x]\!]$ (the ring of power series in one variable over a field $k$), $J = 0$, and $I=(x) = \m$, then clearly $J \subseteq I \subseteq \sqrt{J+\m I}$, but the radical ideals $J$ and $I$ are distinct.

\begin{defn}
Let $R$ be a ring, $\rc$ a closure operation on $R$, and $J \subseteq I$ ideals.  We say that $J$ is a \emph{$\rc$-reduction} of $I$ if $J^\rc = I^\rc$.  A $\rc$-reduction $J \subseteq I$ is \emph{minimal} if for all ideals $K \subsetneq J$, $K$ is not a $\rc$-reduction of $I$.
\end{defn}

For any Nakayama closure $\rc$, one can make a very strong statement about the existence of minimal $\rc$-reductions:
\begin{lemma}\label{lem:red}\cite[Lemma 2.2]{nme*spread}
If $\rcl$ is a Nakayama closure on $R$ and $I$ an ideal, then for any $\rcl$-reduction $J$ of $I$, there is a minimal $\rcl$-reduction $K$ of $I$ contained in $J$. Moreover, in this situation any minimal generating set of $K$ extends to a minimal generating set of $J$.
\end{lemma}

This is another reason to see that radical is not Nakayama, as, for instance, the ideal $(x)$ in the example above has no minimal $\sqrt{\ }$-reductions at all.

\subsection{Special parts of closures}
Consider the following notion from \cite{nme-sp}:

\begin{defn}
Let $(R,\m)$ be a Noetherian local ring and $\rc$ a closure operation on $R$.  Let $\cI$ be the set of all ideals of $R$.  Then a set map $\rc \spc: \cI \ra \cI$ is a \emph{special part of $\rc$} if it satisfies the following properties for all $I, J \in \cI$: \begin{enumerate}
\item (Trapped) $\m I \subseteq I^{\rc \spc} \subseteq I^\rc$,
\item (Depends only on the closure) $(I^\rc)^{\rc \spc} = I^{\rc \spc}$,
\item (Order-preserving) If $J \subseteq I$ then $J^{\rc \spc} \subseteq I^{\rc \spc}$,
\item (Special Nakayama property) If $J \subseteq I \subseteq (J + I^{\rc \spc})^\rc$, then $I \subseteq J^\rc$.
\end{enumerate}
\end{defn}

Of course, any closure operation $\rc$ that admits a special part $\rc \spc$ must be a Nakayama closure.  Examples of special parts of closures include: \begin{itemize}
\item ``Special tight closure'' (first defined by Adela Vraciu in \cite{Vr*ind}), when $R$ is excellent and of prime characteristic $p$: \[
I^{*\spc} := \{f \in R \mid x^{q_0} \in (\m I^{[q_0]})^* \text{ for some power } q_0 \text{ of } p\}.
\]
\item The special part of Frobenius closure (see \cite{nme-sp}), when $R$ has prime characteristic $p$: \[
I^{F\spc}:= \{f \in R \mid x^{q_0} \in \m I^{[q_0]} \text{ for some power } q_0 \text{ of } p \}.
\]
\item The special part of integral closure (see \cite{nme-sp}): \[
I^{-\spc} := \{f \in R \mid x^n \in (\m I^n)^- \text{ for some } n \in \N \}.
\]
\item The special part of plus closure (when $R$ is a domain) (see \cite{nmePhD}): \[
I^{+\spc} := \{f \in R \mid f \in \Jac(S)IS \text{ for some module-finite domain extension } R \ra S\},
\]
where $\Jac(S)$ is the Jacobson radical of $S$
\end{itemize}

All of these are in fact special parts of the corresponding closures.  Moreover, we have the following: \begin{prop}\ \begin{itemize}
\item If $(R,\m)$ has a perfect residue field, then $I^F = I + I^{F\spc}$ for all ideals $I$. \cite{nme-sp}
\item If $(R,\m)$ is a Henselian domain with algebraically closed residue field, then $I^+ = I + I^{+\spc}$ for all ideals $I$. \cite{nmePhD}
\item If $(R,\m)$ is either an excellent, analytically irreducible domain with algebraically closed residue field, or an excellent normal domain with perfect residue field, then $I^* = I+I^{*\spc}$ for all ideals $I$. \cite{HuVr, nme*spread}
\end{itemize}
\end{prop}

This property is called \emph{special $\rc$-decomposition}, and goes back to the theorem of Huneke and Vraciu \cite{HuVr} on tight closure, and K. Smith's theorem \cite[Theorem 2.2]{Sm-tcgraded} as a kind of pre-history. Of course integral closure fails this property rather badly.  Nevertheless, special $\rc$-decompositions allow for the following, matroidal proof of the existence of the notion of ``spread'' in such circumstances:

\begin{thm}
Suppose $\rc$ is a closure with a special part $\rc \spc$.  Suppose the following property holds: \begin{enumerate}
\item[($\alpha$)] $I^\rc = I + I^{\rc \spc}$ for all ideals $I$ of $R$.
\end{enumerate}
Let $I$ be an arbitrary ideal.  Then every minimal $\rc$-reduction of $I$ is generated by the same number of elements: the \emph{$\rc$-spread} $\ell^\rc(I)$ of $I$.
\end{thm}

\begin{proof}
Let $J$, $K$ be minimal $\rc$-reductions of $I$.  Say $\{a_1, \dotsc, a_n\}$ is a minimal generating set for $J$, $\{b_1, \dotsc, b_r\}$ is a minimal generating set for $K$, and $n\leq r$.

\noindent \textbf{Claim:} There is a reordering of the $a_j$'s in such a way that for each $0\leq i \leq n$, we have $I \subseteq L_i^\rc$, where $L_i := (a_1, \dotsc, a_i, b_{i+1}, \dotsc, b_n)$.

\begin{proof}[Proof of Claim]
We proceed by descending induction on $i$.  If $i=n$ there is nothing to prove.  So assume that $I \subseteq L_i^\rc$ for some $1\leq i\leq n$.  We need to show that $I \subseteq L_{i-1}^\rc$.  It suffices to show that $a_i \in L_{i-1}^\rc$, since this would imply that $L_i \subseteq L_{i-1}^\rc$.

Note that $b_i \in I^\rc = L_i^\rc = L_i + L_i^{\rc \spc}$ (by property ($\alpha$)).  Hence, there exist $r_j, s_j \in R$ such that \[
b_i + \sum_{j=i+1}^n r_j b_j + \sum_{j=1}^i s_j a_j \in L_i^{\rc \spc}.
\]
If all the $s_j \in \m$, then \[
b_i + \sum_{j=i+1}^n r_j b_j  \in (\m L_i + L_i^{\rc \spc} = L_i^{\rc \spc} = I^{\rc \spc}) \cap K = K^{\rc \spc} \cap K = \m K,
\]
contradicting the fact that the $b_j$ form a minimal generating set for $K$.  Hence, by reordering the $a_j$, we may assume that $s_i \notin \m$.

Thus, $a_i \in L_{i-1} + L_i^{\rc\spc}$.  It then follows from the special Nakayama property that $a_i \in L_{i-1}^\rc$, as required.
\end{proof}
Applying the Claim with $i=0$ gives that $I \subseteq L_0^\rc = (b_1, \dotsc, b_n)^\rc$.  But since $K$ was a \emph{minimal} $\rc$-reduction of $I$, it follows that $(b_1, \dotsc, b_n) = K = (b_1, \dotsc, b_r)$, whence $n=r$.
\end{proof}

\begin{rmk}
Integral closure does not satisfy quite the same matroidal properties shared by plus closure, tight closure, and Frobenius closure, but it does so `generically' as shown in \cite{nmeBrn-mt}.
\end{rmk}

It is natural to ask the following: Suppose $\rc$ is a closure operation with well-defined $\rc$-spread $\ell^\rc$.  If one picks a random collection of $\ell^c(I)$ elements of $I/\m I$, will their lifts to $I$ generate a minimal $\rc$-reduction of $I$?  This is unknown in general, but at least for integral closure and tight closure, the answer is often ``yes'':
 \begin{prop} Let $I$ be an ideal of the Noetherian local ring $(R,\m)$.  Suppose we are in one of the following situations: \begin{itemize}
\item \cite[Theorem 5.1]{NR} $R$ has infinite residue field, and $\rc =$ integral closure.   
\item \cite[Theorem 4.5]{FoVa-core} $\chr R=p>0$ prime, $R$ is excellent and normal with infinite perfect residue field, and $\rc=$ tight closure.
\end{itemize}
Let $\ell := \ell^\rc(I)$.  Then there is a Zariski-open subset $U \subseteq (I/\m I)^\ell$ such that whenever $x_1, \dotsc, x_\ell \in I$ are such that $(x_1 + \m I, \dotsc, x_\ell + \m I) \in U$, the ideal $(x_1, \dotsc, x_\ell)$ is a minimal $\rc$-reduction of $I$.
\end{prop}

As soon as one knows that minimal $\rc$-reductions exist in the sense given in Lemma~\ref{lem:red} (thus, for any Nakayama closure $\rc$), one can investigate the \emph{$\rc$-$\core$} of an ideal, which is by definition the intersection of all its (minimal) $\rc$-reductions.  When $\rc$=integral closure (in which case we just call this object the \emph{core}), this notion was first investigated by Rees and Sally \cite{ReSa-core}.  They showed that if $R$ is a \emph{regular local ring} of dimension $d$, then for any ideal $I$, $(I^d)^- \subseteq \core(I)$.  Fouli and Vassilev \cite{FoVa-core} have investigated the relationship between the core and the $*$-core of ideals in rings of prime characteristic $p$.

In joint work with Holger Brenner, we are developing a related notion (which one could call ``specific $\rc$-closure'' $\rc \sigma$), which we show in many cases coincides with $\rc \spc$, defined as follows: Let $\rc$ be a closure operation defined by certain properties of generic forcing algebras (see \ref{sub:forcing}).  That is, let $(R,\m)$ be local and for $f\in R$ and an ideal $I = (f_1, \dotsc, f_n)$ of $R$, let $A = R[T_1, \dotsc, T_n] / (f + \sum_{i=1}^n f_i T_i)$.  Consider the ideal $\n := \m A + (T_1, \dotsc, T_n)$ of $A$.  Suppose that $\cP$ is a property of $R$-algebras such that $f \in I^\rc$ if and only if $A$ has property $\cP$.  We say that $f$ is in the \emph{specific $\rc$-closure} of $I$ if $A_\n$ also has property $\cP$.  In the cases where both are defined, we can show that $I^{\rc \spc} \subseteq I^{\rc \sigma}$.  Moreover, we show that $F \spc = F \sigma$ in general, that $+ \spc = +\sigma$ in Henselian domains (e.g. complete local domains), and that $*\spc = \star\sigma$ in normal local domains of prime characteristic with perfect residue field.

\section{Classes of rings defined by closed ideals}\label{sec:rings}

A typical reason that a closure operation is studied in the first place is often that the closedness of certain classes of ideals is related (and often equivalent) to the ring having certain desirable properties.  In the following, we will give examples based on various closure operations.

\subsection{When is the zero ideal closed?}
When $\rc = F, +, *, \star, \sqrt{\ },$ or $-$, we have $0^\rc = \sqrt{0}$, the \emph{nilradical} of the ring.  Hence, in these cases, the zero ideal is closed iff the ring is \emph{reduced}.

For a local ring $(R,\m)$, we have $(0:\m^\infty) = H^0_{\m}(R)$.  So the zero ideal is $\m$-saturated iff $\depth R =\max\{\dim R, 1\}$ (that is, iff $R$ is an (S$_1$) ring).

\subsection{When are 0 and principal ideals generated by non-zerodivisors closed?}\label{sub:princlosed}
\begin{example*}[Radical]
For a Noetherian ring $R$, the answer here is: ``$R$ is a product of fields".  First, note that if $R$ is a product of fields, all ideals are radical.  Conversely, $R$ is a Noetherian ring where  $0$ and all principal ideals generated by non-zerodivisors are radical.  As noted above, the fact that $0$ is radical means that $R$ is reduced.  Now take any non-zerodivisor $z$ of $R$.  We have $z \in \sqrt{(z^2)} = (z^2)$ (since $z^2$ is a non-zerodivisor), so there is some $r\in R$ with $z = rz^2$.  That is, $z(1-rz)=0$.  Since $z$ is a non-zerodivisor, this means that $rz=1$, so that $z$ is a unit.  Hence $R$ is \emph{a reduced ring in which all non-zerodivisors are units}, and thus a product of fields.

When $R$ is reduced, this condition may be written: $R=Q(R)$.
\end{example*}

\begin{example*}[Integral and tight closures] Recall that for reduced rings, tight closure and integral closure coincide for principal ideals.  So in these cases, we may look at integral closure, in which case the answer is: ``$R$ is integrally closed in its total ring of fractions".

For suppose $R$ is integrally closed in its total ring $Q$ of fractions.  Let $f \in R$ be a non-zerodivisor, and suppose $g\in (f)^-$.  Then we have an equation of the form \[
g^n + a_1 f g^{n-1} + \cdots + a_n f^n = 0,
\]
where $a_1, \dotsc, a_n \in R$.  Dividing through by $f^n$, we get the equation: \[
(g/f)^n + a_1 (g/f)^{n-1} + \cdots + a_n = 0,
\]
which shows that the element $g/f\in Q$ is integral over $R$.  But this means that $g/f \in R$, whence $g \in (f)$.  The converse statement follows the same steps in reverse.

For reduced rings, this condition may be written: $R=R^+ \cap Q(R)$.
\end{example*}

\begin{example*}[Frobenius closure]
For a Noetherian ring of characteristic $p>0$, the answer here is: ``$R$ is weakly normal''\footnote{Weak normality has a complicated history.  The reader may consult the recent guide, \cite{Vit-survey}.}.  To see this, recall first \cite[Proposition 1]{Ito-wn} that a Noetherian ring $R$ is weakly normal if and only if the following conditions hold (where $Q$ is the total quotient ring of $R$): \begin{enumerate}
\item $R$ is reduced.
\item For every $f\in Q$ such that $f^2, f^3 \in R$, one has $f\in R$.
\item For every $f\in Q$ such that there exists a prime integer $n$ with $f^n, nf \in R$, one has $f\in R$.
\end{enumerate}
Suppose first that $R$ is weakly normal.  The fact that $R$ is reduced means that $0$ is Frobenius closed.  So let $z$ be a non-zerodivisor of $R$ and suppose $y \in (z)^F$.  Then there is some $a\in R$ and $q=p^e$ such that $y^{p^e}\in (z^{p^e})$ as an ideal of $R$.  In $Q(R)$, this means that $(y/z)^{p^e} = ((y/z)^{p^{e-1}})^p \in R$.  But we also have $p \cdot (y/z)^{p^{e-1}} = 0 \in R$, so by condition (3) above, it follows that $(y/z)^{p^{e-1}} \in R$, i.e., $y^{p^{e-1}} \in (z^{p^{e-1}})$.   By induction on $e$, it follows that $y \in (z)$.

Conversely, suppose that $0$ and every principal ideal generated by a non-zerodivisor are Frobenius closed.  Condition (1) follows from $0$ being Frobenius closed.  For condition (2), note that there exist integers $a,b\in \N$ such that $2a+3b=p$.  So if $f^2, f^3 \in R$, it follows that $f^p = (f^2)^a (f^3)^b \in R$ as well.  But $f=y/z$ for some $y,z \in R$, where $z$ is a non-zerodivisor, so $f^p \in R$ means that $y^p \in (z^p)$, so that $y \in (z)^F = (z)$, whence $f\in R$.  As for condition (3), let $f$, $n$ be as given.  If $n\neq p$, then the image of $n$ in $R$ is a unit, so that the fact that $nf \in R$ means that $f\in R$ automatically.  So we may assume that $n=p$.  But then $f^p \in R$, which means that $y^p \in (z^p)$ (where $f=y/z$ for some non-zerodivisor $z$), whence $y \in (z)^F = (z)$ and $f\in R$.
\end{example*}

\subsection{When are parameter ideals closed (where $R$ is local)?}
For this subsection and the next, we need the following definitions (which represent a special case of ``reduction to characteristic $p$''):

\begin{defn}\label{def:type}
Let $\cP$ be a property of positive prime characteristic rings.  Let $R$ be a finitely generated $K$-algebra, where $K$ is a field of characteristic $0$.  That is, $R = K[X_1, \dotsc, X_n] / (f_1, \dotsc, f_m)$.  Let $A$ be the $\Z$-subalgebra of $K$ generated by the \emph{coefficients} of the polynomials $f_1, \dotsc, f_m$.  Let $R_A := A[X_1, \dotsc, X_n] / (f_1, \dotsc, f_m)$.  We say that $R$ is \emph{of $\cP$-type} (resp. \emph{dense $\cP$-type}) if there is a nonempty Zariski-open subset $U$ of the maximal ideal space of $A$ (resp. an infinite set $U$ of maximal ideals of $A$) such that for all $\mu \in U$, the ring $R_A \otimes_A A/\mu$ has property $\cP$.
\end{defn}

\begin{example*}[Frobenius closure]
Suppose $(R,\m)$ is a local \emph{Cohen-Macaulay} ring of prime characteristic $p>0$.  Consider the map $R \ra {}^1R$ defined in \ref{sub:fromcon}(\ref{it:Frob}).  By definition $R$ is \emph{$F$-injective} if the induced maps on local cohomology $H^i_\m(R) \ra H^i_\m({}^1R)$ are injective.  All parameter ideals are Frobenius closed if and only if $R$ is $F$-injective.

Moreover, there are connections to characteristic 0 singularity theory.  Namely, if $R$ is a reduced, finitely generated $\Q$-algebra (or $\C$-algebra) of dense $F$-injective \emph{type}, then $\Spec R$ has only \emph{Du Bois singularities} \cite{Sc-FiDB}.  (A converse is not known.)
\end{example*}

\begin{example*}[Tight closure]
By definition, a ring $R$ of prime characteristic $p>0$ is \emph{$F$-rational} if every parameter ideal is tightly closed.  

The connection to characteristic $0$ singularity theory is very strong.  Indeed, it was shown by \cite{SmFrat, Harat} (Smith showed one direction and Hara showed the converse) that a finitely generated $\Q$-algebra (or $\C$-algebra) $R$ is of $F$-rational \emph{type} if and only if $\Spec R$ has only \emph{rational singularities}.
\end{example*}

\begin{example*}[Integral closure]
Let $(R,\m)$ be a Noetherian local ring.  Then every parameter ideal is integrally closed if and only if $R$ is either a field or a rank 1 discrete valuation ring.

As usual, one direction is clear: if $R$ is a field or a rank 1 DVR, then \emph{every} ideal is integrally closed.  So we prove the converse.  We already know that $R$ must be a normal domain (from the integral closure example in \ref{sub:princlosed}), so we need only show that $\dim R \leq 1$.  Accordingly, suppose $\dim R \geq 2$.  Let $x,y$ be part of a system of parameters.  Then $x^2, y^2$ is also part of a system of parameters. Thus, we have \[
xy \in (x^2, y^2)^- = (x^2, y^2),
\]
But this is easily seen to contradict the fact that $x,y$ are a regular sequence (which in turn arises from the fact that $R$ is normal, hence (S$_2$).)
\end{example*}

\subsection{When is every ideal closed?}
\begin{example*}[Frobenius closure]
For a Noetherian ring $R$ of prime characteristic, we say that $R$ is \emph{$F$-pure} if the map $R \ra {}^1R$ defined in \ref{sub:fromcon}(\ref{it:Frob}) is a \emph{pure} map of $R$-modules -- that is, when tensored with any other $R$-module, the resulting map is always injective.  This is easily seen to be equivalent to the condition that Frobenius closure on \emph{modules} is trivial, which  for excellent rings is then equivalent  to the condition that all ideals are Frobenius closed (by \cite[Theorem 1.7]{Ho-purity}, since the rings in question must be reduced as $(0)^F$ is the nilradical of the ring).

Again, there is a connection to characteristic 0 singularity theory.  Namely, if $R$ is normal, $\Q$-Gorenstein, and characteristic 0 of dense $F$-pure type, then $\Spec R$ has only \emph{log canonical singularities} \cite{HaWa-Fsing}. (No converse is known.)
\end{example*}

\begin{example*}[Tight closure]
For a Noetherian ring $R$ of prime characteristic, we say that $R$ is \emph{weakly $F$-regular} if every ideal of $R$ is tightly closed.  It is called \emph{$F$-regular} if $R_\p$ is weakly $F$-regular for all prime ideals $\p$. (Or equivalently, $W^{-1}R$ is weakly $F$-regular for all multiplicative sets $W \subset R$.)  One of the outstanding open problems of tight closure theory is whether these two concepts are equivalent.  The problem is open even for finitely generated algebras over a field.

In any case, we again have a connection to characteristic 0 singularity theory.  Namely, if $R$ is normal, $\Q$-Gorenstein, and characteristic 0, then $R$ has $F$-regular \emph{type} if and only if $\Spec R$ has only \emph{log terminal singularities} \cite{HaWa-Fsing}.
\end{example*}

\begin{example*}[Integral closure]
As shown earlier, assuming that $R$ is Noetherian, all ideals are integrally closed iff $R$ is a Dedekind domain.  (More generally, all ideals are integrally closed iff $R$ is a \emph{Pr\"ufer domain}.)
\end{example*}

\begin{example*}[The $\rv$-operation]
Let $R$ be a Noetherian domain.  Then \cite[Theorem 3.8]{Mat-refl} all ideals are $\rv$-closed (equivalently $\rt$-closed) if and only if $R$ is Gorenstein of Krull dimension 1 (iff $Q/R$ is an injective $R$-module, where $Q$ is the total quotient ring of $R$).
\end{example*}

\section{Closure operations on (sub)modules}\label{sec:modules}
Whenever an author writes about a closure operation on (ideals of) rings that extends to a closure operation on (sub)modules, a choice must be made.  Either the operation is defined generally for all submodules, after which all proofs and statements are made for modules and the statements about ideals follow as  a corollary, or everything is done first exclusively on ideals and only afterwards is the reader shown how to extend it to the module case.  I have opted for the latter option.

\begin{defn}
Let $R$ be a ring and $\cM$ a category of $R$-modules.  A \emph{closure operation on $\cM$} is a collection of maps $\{\rcl_M \mid M \in \cM\}$, such that for any submodule $L \subseteq M$ of a module $M \in \cM$, $\rcl_M(L) := L^\rcl_M$ is a submodule of $M$, such that the following properties hold:\begin{enumerate}
\item (Extension) $L \subseteq L^\rcl_M$ for all submodules $L \subseteq M \in \cM$.
\item (Idempotence) $L^\rcl_M = (L^\rcl_M)^\rcl_M$ for all submodules $L \subseteq M \in \cM$.
\item (Order-preservation) If $M \in \cM$ and $K \subseteq L \subseteq M$ are submodules, then $K^\rcl_M \subseteq L^\rcl_M$.  If moreover $L\in \cM$, then $K^\rcl_L \subseteq K^\rcl_M$.
\end{enumerate}
\end{defn}

 Some other properties of closure operations on modules are given below:
 \begin{defn}
 Let $\rcl$ be a closure operation on a category $\cM$ of $R$-modules.  We say that $\rcl$ is \begin{enumerate}
 \item \emph{functorial} if whenever $g: M \ra N$ is a morphism in $\cM$, then for any submodule $L \subseteq M$, we have $g(L^\rcl_M) \subseteq g(L)^\rcl_N$.
 \item \emph{semi-prime} if whenever $L \subseteq M$ are modules with $M \in \cM$ and $I$ is an ideal, we have $I \cdot L^\rcl_M \subseteq (IL)^\rcl_M$.
 \item \emph{weakly hereditary} if whenever $L \subseteq M \subseteq N$ are submodules such that $M,N \in \cM$, if $L = L^\rcl_M$ and $M = M^\rcl_N$ then $L = L^\rcl_N$.
 \item \emph{hereditary} if whenever $L \subseteq M \subseteq N$ are submodules such that $M, N \in \cM$, then $L^\rcl_N \cap M = L^\rcl_M$.
  \item \emph{residual} if whenever $\pi: P \twoheadrightarrow M$ is a surjection of modules in $\cM$, and $N \subseteq P$ is a submodule, we have \[
 N^\rcl_P =\pi^{-1}\left(\pi(N)^\rcl_M\right)
 \]
  \end{enumerate}
 \end{defn}
 
\begin{rmk}\footnote{Thanks to Holger Brenner for pointing this out.}
We note here that any functorial closure is semi-prime, provided that for any $x\in R$ and $M \in \cM$, the endomorphism $\mu_x: M \ra M$ (defined by multiplying by the element $x$) is in $\cM$.  To see this, let $L \subseteq M$ be a submodule and $x \in I$.  Then $x \cdot L^\rcl_M = \mu_x(L^\rcl_M) \subseteq (\mu_x(L))^\rcl_M = (xL)^\rcl_M$.  But by order-preservation, $(xL)^\rcl_M \subseteq (IL)^\rcl_M$ for any $x\in I$.  Thus, we have \[
I \cdot L^\rcl_M = \sum_{x\in I} x \cdot L^\rcl_M \subseteq \sum_{x\in I} (x L)^\rcl_M \subseteq (IL)^\rcl_M.
\]

For systematic reasons, one usually assumes that a closure operation is functorial.  The disadvantage, of course, is that there is no way to extend a non-semi-prime closure on $R$ to a functorial module closure on $\cM$ (e.g. the $\rv$-operation on certain non-domains, as in Proposition~\ref{pr:vop}(3)). 
\end{rmk}
 
\begin{rmk}
Given a closure operation on (ideals of) $R$ that one wants to extend to a (functorial) closure on a category $\cM$ of $R$-modules, one typically must make a choice: should the resulting operation be \emph{weakly hereditary} or \emph{residual}?  Most closure operations on modules do not satisfies both conditions (but see \ref{sub:torth} below).  

Among integral closure-theorists, the usual goal has been to construct a \emph{weakly hereditary} operation (so that for instance, if $J$ is a reduction of an integrally closed ideal $I$, then the integral closure of $J$ as a \emph{submodule} of  $I$ is required to be $I$ itself).  See \cite{EHU-Ralg}, or \cite[chapter 16]{HuSw-book} for a general discussion.  

On the other hand, tight closure-theorists have gone the \emph{residual} route (so that for instance, to compute the tight closure of an ideal $I$, one may take the tight closure of $0$ in the $R$-submodule $R/I$ and contract back).  See \cite[Appendix 1]{HuTC}.  For a \emph{residual} version of integral closure, see \cite{nmeUlr-lint}, by the present author and Bernd Ulrich.
\end{rmk} 

\begin{lemma}[Taken from \cite{nmeHo-heq}]
Let $\cF$ (resp. $\cF_{\rm fin}$) be the category of free (resp. finitely generated free) $R$-modules.  Let $\rcl$ be a functorial closure operation on $\cF$ (resp. $\cF_{\rm fin}$).  Then $\rcl$ extends uniquely to a \emph{residual} closure operation on the category of all (resp. all finitely generated) $R$-modules if and only if for all (resp. for all finitely generated) $R$-modules $F_1, F_2$ and submodules $L_1 \subseteq F_1$, $L_2 \subseteq F_2$, we have \[
(L_1 \oplus L_2)^\rcl_{F_1 \oplus F_2} = (L_1)^\rcl_{F_1} \oplus (L_2)^\rcl_{F_2}.
\]
\end{lemma}

We omit the easy proof.  As for the construction: For a (f.g.) $R$-module $M$, let $\pi: F \ra M$ be a surjection from a (f.g.) $R$-module $F$, and let $L^\rcl_M := \pi(\pi^{-1}(L)^\rcl_F)$.

Indeed, this is the way that tight closure, Frobenius closure, and plus closure are defined on modules.  First they were defined for free modules, and then extended to a residual operation in the way outlined above.

\begin{rmk}\label{rmk:residualexamples}
All but one of the constructions from \S\ref{sec:cons} can be extended to make \emph{residual} closure operations.  Namely: \begin{itemize}
\item Construction~\ref{con:module} may be extended as follows: Fix an $R$-module $U$ as before, and let $L \subseteq M$ be $R$-modules.  We say that $f \in L^\rcl_M$ if the image of the map $U \otimes_R (Rf \hookrightarrow M)$ is contained in the image of the map $U \otimes_R (L \hookrightarrow M)$.  This always yields a residual and functorial closure operation.
\item Construction~\ref{con:contraction} may be extended similarly.  That is, if $\phi:R \ra S$ is a ring homomorphism and $\rd$ is a closure operation on $S$-modules, then for $R$-modules $L \subseteq M$, let $j_M: M \ra S \otimes_R M$ be the natural map, and for $z\in M$, we declare that $z\in L^\rc_M$ if the image (under $j_M$) of $z$ is in the $\rd$-closure (in $S \otimes_R M$) of the image of the map $S \otimes_R (L \hookrightarrow M)$.  Then $\rc$ is a closure operation, semi-prime if $\rd$ is, functorial if $\rd$ is, and residual if $\rd$ is.
\item Constructions~\ref{con:intersection}, \ref{con:directunion}, and \ref{con:idemhull} may be extended in the obvious way to modules, and the resulting operation is semi-prime (resp. residual, resp. functorial) if the operations it is based on are.
\end{itemize}
In this way (via \ref{sub:fromcon}), all the closure operations from Example~\ref{ex:closures} may be extended to residual, functorial closure operations on modules.  This yields the usual definitions of Frobenius, plus, solid, tight, and basically full closures, and also of the $\ia$-saturation.  The resulting extension of integral closure is of course not the usual (weakly hereditary) definition, but rather the residual, ``liftable integral closure'' from \cite{nmeUlr-lint}.  The extension of radical is the one mentioned in \cite{Br-Groth}.

An analogue of Proposition~\ref{pr:prdec}(parts 1 through 4) is also available for semi-prime closure operations on modules.
\end{rmk}

One more construction should be mentioned:
\begin{construction}\label{con:finitistic}
Let $\rc$ be a closure operation on a module category $\cM$.  We define the \emph{finitistic} $\rc$-closure of a submodule $L \subseteq M$ by setting \[
L^{\rc\fg}_M := \bigcup \{L^\rc_N \mid L \subseteq N \subseteq M \text{ such that }  N/L \text{ is finitely generated and } N \in \cM\}
\]
\end{construction}

For a closure operation $\rc$, one may ask what the difference is, if any, between the closure operations $\rc$ and $\rc\fg$.  This is very important in tight closure theory, for instance, in the question of whether weak $F$-regularity and $F$-regularity coincide.  Indeed, if $(R,\m)$ is an $F$-finite Noetherian local ring of characteristic $p$ and $E$ is the injective hull of the residue field $R/\m$, then $R$ is weakly $F$-regular if and only if $0^{*\fg}_E = 0$ \cite[Proposition 8.23]{HHmain}, but it is $F$-regular if $0^*_E = 0$ \cite[follows from Proposition 2.4.1]{LySmFreg}.  Thus, the two concepts would coincide if $0^*_E = 0^{*\fg}_E$.

\subsection{Torsion theories}\label{sub:torth}

A functorial closure operation on modules that is both residual and weakly hereditary is essentially equivalent to the notion of a \emph{torsion theory} (see \cite{Dic-torcat} for torsion theories, the book \cite{BKN-preradbook} for the connection with preradicals, or \cite[section 6]{Br-Groth} for a connection with Grothendieck topologies).  We briefly outline the situation below, as this provides a fertile source of residual closure operations:

\begin{defn}
Let $\cM$ be a category of $R$-modules which is closed under taking quotient modules, and let $\br: \cM \ra \cM$ be a subfunctor of the identity functor (That is, it assigns to each $M \in \cM$ a submodule $\br(M) \in \cM$ of $M$, and every map $g: M \ra N$ in $\cM$ restricts to a map $\br(g): \br(M) \ra \br(N)$.)  Then we say that $\br$ is a \emph{preradical}.

A preradical $\br$ is called a \emph{radical} if for any $M \in \cM$, we have $\br(M/\br(M)) = 0$.

A radical $\br$ is \emph{idempotent} if $\br \circ \br = \br$.

A radical $\br$ is \emph{hereditary} if for any submodule inclusion $L \subseteq M$ with $L, M \in \cM$, we have $L \cap \br(M) = \br(L)$.
\end{defn}

\begin{defn}
Let $\cM$ be an Abelian category.  Then a \emph{torsion theory} is a pair $(\cT, \cF)$ of classes of objects of $\cM$ (called the \emph{torsion} objects and the \emph{torsion-free} objects, respectively), such that: \begin{enumerate}
\item $\cT \cap \cF = \{0\}$,
\item $\cT$ is closed under isomorphisms and quotient objects,
\item $\cF$ is closed under isomorphisms and subobjects, and
\item For any object $M \in \cM$, there is a short exact sequence \[
0 \ra T \ra M \ra F \ra 0
\]
with $T \in \cT$ and $F \in \cF$.
\end{enumerate}
(Recall that the $T$ and $F$ in the short exact sequence are unique up to isomorphism.)

A torsion theory $(\cT, \cF)$ is \emph{hereditary} if $\cT$ is also closed under taking subobjects.
\end{defn}

\begin{prop}
Given an Abelian category $\cM$ of $R$-modules, the following structures are equivalent: \begin{itemize}
\item a functorial, residual closure operation on $\cM$.
\item a radical on $\cM$.
\end{itemize}
Moreover, the closure operation is \emph{weakly hereditary} if and only if the radical is \emph{idempotent}.  In this case, the structures are equivalent to specifying a \emph{torsion theory} on $\cM$.  In this case, the closure operation is hereditary iff the radical is hereditary iff the torsion theory is hereditary.
\end{prop}

\begin{proof}
Instead of a complete proof, we show the correspondences below and leave the elementary proofs to the reader.

If $\rc$ is a functorial, residual closure operation on $\cM$, we define a radical $\br$ on $\cM$ by letting $\br(M) := 0^\rc_M$ for any $M \in \cM$.

If $\br$ is a radical on $\cM$, we define a closure operation on $\cM$ by letting $L^\rc_M := \pi^{-1}(\br(M/L))$ for any $L \subseteq M$, where $\pi: M \ra M/L$ is the canonical surjection..

Given an idempotent radical $\br$ on $\cM$, we define a torsion theory by letting $\cT := \{\br(M) \mid M \in \cM\}$ and $\cF := \{M \in \cM \mid \br(M) =0\}$.  Conversely, given a torsion theory $(\cT, \cF)$, we let $\br(M) := T$, where $0 \ra T \ra M \ra F \ra 0$ is a short exact sequence with $T \in \cT$ and $F \in \cF$.
\end{proof}

The literature on torsion theories is immense.  At the time of publication, I could not find the first place where it is shown that specifying a \emph{hereditary} radical is equivalent to specifying a certain kind of filter of ideals, but this is nevertheless true, and was known by the late 1960s.

None of tight closure, plus closure, Frobenius closure, or radical (as given in Remark~\ref{rmk:residualexamples}) are weakly hereditary; hence none of them provide examples of torsion theories. However, we have the following illustrative examples:

\begin{itemize}
\item For any ideal $\ia$ of $R$, the \emph{$\ia$-saturation} (defined by $L^\rc_M := (L :_M \ia^\infty) = \{z \in M \mid \exists n \in \N \text{ such that } \ia^n z \subseteq L\}$) is a hereditary, residual, functorial closure operation on $R$-modules, and thus provides an example of a hereditary torsion theory on $R$-modules.  (The corresponding radical is $H^0_\ia(-)$.)
\item Let $R$ be a commutative ring that contains at least one non-zerodivisor $x$ that is not a unit, and let $Q$ be its total ring of quotients (so that in particular, $Q \neq R$).  Recall that a module $M$ is \emph{divisible} if for every non-zerodivisor $r$ of $R$, the map $M \ra M$ given by multiplication with $r$ is \emph{surjective}.  Any module has a unique \emph{largest} divisible submodule, since any sum of divisible submodules of a module is divisible.  Consider the assignment $\bd$ given by $\bd(M) := $ the largest divisible submodule of $M$.  Then this is a weakly hereditary radical that is not hereditary.  Hence, it defines a weakly hereditary residual closure operation that is not hereditary, hence a non-hereditary torsion theory.  To see that it is not hereditary, note that $\bd(Q) \cap R = R$ but $1 \notin \bd(R)$ since $x$ is not a unit.
\end{itemize}

\section*{Acknowledgements}
The author (and therefore the paper itself) benefitted from conversations with Holger Brenner, Karl Schwede, and Janet Vassilev, among others, as well as from the referee's careful comments.  Many thanks!

%\bibliographystyle{amsalpha}
%\bibliography{rsch_refs}
\providecommand{\bysame}{\leavevmode\hbox to3em{\hrulefill}\thinspace}
\providecommand{\MR}{\relax\ifhmode\unskip\space\fi MR }
\providecommand{\MRhref}[2]{%
  \href{http://www.ams.org/mathscinet-getitem?mr=#1}{#2}
}
\providecommand{\href}[2]{#2}

\end{document}